\documentclass[11pt]{amsart}

\usepackage{amsthm, amsfonts, amssymb, amscd}
\usepackage[pagebackref,colorlinks]{hyperref}

\theoremstyle{definition}
\newtheorem{ntn}{Notation}[section]
\newtheorem{dfn}[ntn]{Definition}
\theoremstyle{plain}
\newtheorem{lem}[ntn]{Lemma}
\newtheorem{prp}[ntn]{Proposition}
\newtheorem{thm}[ntn]{Theorem}
\newtheorem{cor}[ntn]{Corollary}

\theoremstyle{remark}

\newtheorem{rem}[ntn]{Remark}
\newtheorem{exa}[ntn]{Example}

\numberwithin{equation}{section}

\newcommand{\z}{\mathbb{Z}}
\newcommand{\q}{\mathbb{Q}}

\newcommand{\PP}{\mathbb{P}}
\newcommand{\F}{\mathbb{F}}

\newcommand{\EE}{\mathcal{E}}
\newcommand{\Aa}{\mathcal{A}}
\newcommand{\Bb}{\mathcal{B}}

\newcommand{\stabe}{{\rm Stab}}
\newcommand{\GL}{\mathit{{\rm GL}}}

\newcommand{\SL}{\mathit{{\rm SL}}}
\newcommand{\GM}{\mathit{{\rm GM}}}

\newcommand{\Aff}{{\rm Aff}}

\newcommand{\ppp}{\mathfrak{p}}
\newcommand{\mmm}{\mathfrak{m}}

\renewcommand{\H}{\tilde{H}}
\newcommand{\fff}{{R^\times}}
\newcommand{\rr}{{R^\times}}

\newcommand{\tors}{{{\rm Tor}_1^{\z}}}

\newcommand{\CC}{\tilde{C}}

\newcommand{\mt}{\mapsto}
\newcommand{\lan}{\langle}
\newcommand{\ran}{\rangle}
\newcommand{\se}{\subseteq}
\newcommand{\arr}{\rightarrow}
\newcommand{\larr}{\longrightarrow}

\newcommand{\two}{\twoheadrightarrow}

\renewcommand{\char}{{\rm char}}
\newcommand{\diag}{{\rm diag}}

\newcommand{\im}{{\rm im}}
\newcommand{\ind}{{\rm ind}}

\newcommand{\inc}{{\rm inc}}
\newcommand{\id}{{\rm id}}

\newcommand{\corr}{{\rm cor}}
\newcommand{\res}{{\rm res}}
\newcommand{\inv}{{\rm inv}}
\newcommand{\Conj}{{\rm Conj}}

\newcommand {\mtx}[4]
{\left(\!\!\!
\begin{array}{cc}
#1 & #2   \\
#3 & #4
\end{array}
\!\!\!\right)}

\newcommand {\mtxx}[4]
{\left(\!\!
\begin{array}{cc}
\!\!#1 & \!\!#2   \\
\!\!#3 & \!\!#4
\end{array}
\!\!\!\right)}

\newtheoremstyle{athm}
  {}
  {}
  {\itshape}
  {}
  {\scshape}
  {}
  {.5em}
  {\thmnote{#3}}
\theoremstyle{athm}

\begin{document}

\title[Bloch-Wigner theorem]{A Bloch-Wigner exact sequence over local rings}
\author{Behrooz Mirzaii}
\begin{abstract}
In this article we extend the Bloch-Wigner exact sequence over local rings,
where their residue fields have more than nine elements. Moreover, we prove
Van der Kallen's theorem on the presentation of the second $K$-group of  
local rings such that their residue fields have more than four elements. Note that
Van der Kallen proved this result when the residue fields have more than five 
elements.
Although we prove our results over local rings, all our proofs also work over
semilocal rings where all their residue fields have similar properties as
the residue field of local rings.
\end{abstract}
\maketitle

%%%%%%%%%%%%%%%%%%%%%%%%%%%%%%%%%%%%%%%%%%%%%%%%%%%%%%%%%%%%%%%%%%%%%%%%%%%%%%%%
\section*{Introduction}
%%%%%%%%%%%%%%%%%%%%%%%%%%%%%%%%%%%%%%%%%%%%%%%%%%%%%%%%%%%%%%%%%%%%%%%%%%%%%%%%

The Bloch-Wigner exact sequence studies the second and the third $K$-groups of a field. 
On the one hand, it gives Matsumoto's theorem on the presentation of the second $K$-group 
and on the other hand it gives a precise description of the indecomposable part
of the third $K$-group of the field.

This result, proved by Bloch and Wigner independently and in somewhat
different form, in one of its early forms gives the exact sequence
\[
0 \arr \q/\z \arr H_3(\SL_2(k),\z)  \arr  \ppp(k) \arr k^\times\wedge k^\times \arr 
H_2(\SL_2(k),\z)  \arr 0,
\]
where $k$ is an algebraically closed field of characteristic zero \cite{bloch2000},
\cite{dupont-sah1982}. Here $\ppp(k)$ is the 
pre-Bloch group of $k$, which is the free abelian group generated by the symbols $[a]$, 
$a \in k^\times -\{1\}$ up to defining relations
\[
[a]-[b]+\bigg[\frac{b}{a}\bigg]-\bigg[\frac{1-a^{-1}}{1-b^{-1}}\bigg]+\bigg[\frac{1-a}{1-b}\bigg]
\]
and the map $\ppp(k) \arr k^\times\wedge k^\times$ is given by $[a]\mt a\wedge (1-a)$. 
Moreover, the map  $k^\times\wedge k^\times=H_2(k^\times, \z) \arr H_2(\SL_2(k),\z)$
is induce by the map $k^\times \arr \SL_2(k)$, $a\mt \diag(a,a^{-1})$. 
For description of other maps involved in the above sequence see \cite[App. A]{dupont-sah1982}.

Using the homology stability theorem for the general or special linear groups of $k$, 
one can prove (see \cite{sah1989} or \cite{mirzaii-2008}) that 
\[
K_2(k)\simeq H_2(\SL_2(k),\z)\ \ \ \text{ and } \ \ \ K_3^\ind(k)\simeq H_3(\SL_2(k),\z).
\]
%(see \cite{sah1989} or \cite{mirzaii-2008}).
Note that $K_3^\ind(k)$ is the indecomposable part of the third $K$-group of $k$, which is 
the cockerel of the natural map from the third Milnor $K$-group $K_3^M(k)$ to the third 
$K$-group $K_3(k)$. Therefore the Bloch-Wigner exact sequence finds the following form
\[
0 \arr \q/\z \arr K_3^\ind(k) \arr \ppp(k) \arr k^\times\wedge k^\times \arr K_2(k) \arr 0.
\]

This exact sequence had many important applications and was the source of many deep 
ideas in algebraic $K$-theory. Thus it was very important to generalize it to a wider class 
of rings. In a remarkable paper, Suslin has generalized this exact sequence to all infinite 
fields. In fact, he showed that for any infinite field $F$ we have the Bloch-Wigner exact sequence
\[
0 \arr \tors(\mu(F),\mu(F))^\sim \arr K_3^\ind(F)  \arr  \ppp(F) \arr 
(F^\times\otimes_\z F^\times)_\sigma \arr K_2(F)  \arr 0,
\]
where the group $\tors(\mu(F),\mu(F))^\sim$ is the unique nontrivial extension of 
$\tors(\mu(F),\mu(F))$ by $\z/2$ and $(F^\times\otimes_\z F^\times)_\sigma: 
=F^\times\otimes_\z F^\times/\lan a\otimes b + b \otimes a: a, b \in F^\times\ran$.
In \cite{hutchinson2013} Hutchinson showed that the above exact sequence also holds over 
finite fields with more than three elements. His proof is different than Suslin's original 
proof and only works for finite fields.

This article should be seen as sequel to \cite{mirzaii2011}, \cite{mirzaii-mokari2015}, 
where a Bloch-Wigner exact sequence has been proved over rings with many units, e.g.
local or semilocal rings whose their residue fields are infinite.   
In this article, we will extend the Bloch-Wigner exact sequence over local rings,
where their residue fields have more than nine elements. But before proving this, we prove
Van der Kallen's generalization of Matsumoto's theorem on the presentation of the 
second $K$-group of local rings. 

More precisely, first we prove that if $R$ is a local ring with maximal ideal $\mmm_R$
and $R/\mmm_R$ has more than four elements, then we have the exact sequence 
\[
\ppp(R) \arr (\rr \otimes_\z \rr)_\sigma \arr K_2(R) \arr 0.
\]
This immediately implies that $K_2(R)\simeq K_2^M(R)$ (Proposition~\ref{mir-stability}),
where for us
\[
K_2^M(R):=\rr\otimes_\z \rr/\lan a \otimes (1-a), b \otimes (-b) \mid a, 1-a, b \in R^\times \ran.
\]
Note that when $R/\mmm_R$ has more than five elements, the term 
$b \otimes (-b)$ can be removed from the definition of $K_2^M(R)$ (Lemma~\ref{vdk00}).

Furthermore, we prove a Bloch-Wigner exact sequence over local rings. 
Let $|R/\mmm_R|>9$ and $|R/\mmm_R|\neq16,32$. If $R$ is a domain or is 
an algebra over a field we may only assume that $|R/\mmm_R|> 9$.
Then we will prove that we have the exact sequence
\[
T_R \arr K_3^\ind(R) \arr \ppp(R) \arr (\rr \otimes_\z \rr)_\sigma \arr K_2(R)\arr 0,
\]
where $T_R$ sits in the short exact sequence
\[
0 \arr \tors(\mu(R),\mu(R))_\sigma \arr T_R \arr
H_1(\Sigma_2, \mu_{2^\infty}(R)\otimes_\z\mu_{2^\infty}(R)) \arr 0.
\]
Moreover, if there is a homomorphism $R \arr F$, $F$ a field, such that the map
$\mu(R) \arr \mu(F)$ is injective, e.g. $R$ is a domain, then we have the exact sequence
\[
0 \arr \tors(\mu(R), \mu(R))^\sim \arr K_3^\ind(R)\arr \ppp(R) \arr (\rr \otimes_\z \rr)_\sigma 
\arr K_2(R) \arr 0,
\]
where the composition $\tors(\mu(R),\mu(R)) \arr \tors(\mu(R), \mu(R))^\sim \arr K_3^\ind(R)$ 
is induced by the map $\mu(R) \arr \SL_2(R)$,  $\xi \mt \diag(\xi, \xi^{-1})$.

We also prove Hutchinson's Bloch-Wigner exact sequence over finite 
fields \cite{hutchinson2013}. Let $F$ be a finite filed with at least
four elements. Then we prove that if $F\neq \F_4,\F_8$, then we have the exact sequence
\[
0 \arr \tors(\mu(F),\mu(F))^\sim \arr H_3(\SL_2(F))_{F^\times} \arr B(F) \arr 0,
\]
and if $F=\F_4$ or $\F_8$, then we have the exact sequence
\[
0 \arr \tors(\mu(F),\mu(F))\oplus \z/2 \arr H_3(\SL_2(F))_{F^\times} 
\arr B(F) \arr 0,
\]
where $B(F)$ is the Bloch group of $F$, i.e. 
$B(F):=\ker(\ppp(F) \arr (F^\times \otimes_\z F^\times)_\sigma)$.
Moreover, we show that if $F\neq \F_4,\F_8$, then
$H_3(\SL_2(F))_{F^\times}\simeq  K_3^\ind(F)$ and if $F=\F_4$ or $\F_8$,
then $H_3(\SL_2(F))_{F^\times}\simeq  K_3^\ind(F)\oplus \z/2$ and thus
we have the Bloch-Wigner exact sequence
\[
0 \arr \tors(\mu(F), \mu(F))^\sim \arr K_3^\ind(F) \arr B(F) \arr 0.
\]

For the proof of the above results we need certain strong homology stability results for the 
second and the third homology of general linear groups. Although homology stability results with
sharp stability bound is well-known for general linear groups of local rings with infinite residue 
fields \cite{nes-suslin1990}, we couldn't find such a results over local rings with finite 
residue fields. In Sections \ref{H2} and \ref{H3} we will prove certain stability results that
are good enough for our main applications (Theorems \ref{stability1}, \ref{sah-stability}). 

%The proof of our main results that we present here is similar to those given in 
%\cite{mirzaii2011} and \cite{mirzaii-mokari2015}. Here we have to be very careful 
%with the residue field of the local ring. 

It is worth to mention that almost all the results of this article are valid if we replace 
the local ring $R$ with a semilocal ring such that all its residue fields has the same 
property that $R/\mmm_R$ has.

In this paper we shall assume throughout that $R$ is a commutative local ring with 
maximal ideal $\mmm_R$ unless explicitly stated to the contrary. Moreover by the 
homology group $H_n(G)$ we will mean the homology of the group $G$ with integral 
coefficients, i.e. $H_n(G):=H_n(G,\z)$.

%%%%%%%%%%%%%%%%%%%%%%%%%%%%%%%%%%%%%%%%%%%%%%%%%%%%%%%%%%%%%%%
\section{The main spectral sequence}\label{SP}
%%%%%%%%%%%%%%%%%%%%%%%%%%%%%%%%%%%%%%%%%%%%%%%%%%%%%%%%%%%%%%%

Let $R$ be a commutative local ring with maximal ideal $\mmm_R$.
Let $C_l(R^2)$ be the free abelian group generated by
the set of all $(l+1)$-tuples 
$(\lan v_0\ran, \dots, \lan v_l\ran)$, where every $v_i \in R^2$ is a
basis of a direct summand of $R^2$ and any two disjoint vectors
$v_i, v_j$ are a basis of $R^2$. 
%Here by $\lan v_i\ran$ we mean the submodule 
%of $R^2$ generated by $v_i$, i.e. $\lan v_i\ran=Rv_i$. 
%The group $\GL_2(R)$ acts on $X_l(R^2)$ in a natural way. 
%
%Let $C_l(R^2)$ be the free abelian group generated by the set $X_l(R^2)$.
We consider $C_{l}(R^2)$ as a left $\GL_2(R)$-module in a natural way.
If necessary, we convert this action to a right action by the definition 
$m.g:=g^{-1}m$. Let us define the $l$-th differential operator
\[
\partial_l : C_l(R^2) \arr C_{l-1}(R^2), \ \ l\ge 1,
\]
as an alternating sum of face operators which throws away the $i$-th component 
of generators. Hence we have the complex
\begin{displaymath}
C_\bullet(R^2): \ 
\cdots \larr  C_2(R^2) \overset{\partial_2}{\larr}
C_1(R^2) \overset{\partial_1}{\larr} C_0(R^2) \arr 0.
\end{displaymath}
Let $\partial_{-1}=\epsilon: C_0(R^2) \arr \z$ be defined by
$\sum_i n_i(\lan v_{0,i}\ran) \mt \sum_i n_i$.
It is easy to see that $C_1(R^2) \arr C_0(R^2) \arr \z \arr 0$ is exact and
thus $H_0(C_\bullet(R^2))=\z$.

Let $G$ be a group and $L_\bullet$ a complex of left $G$-modules: 
\[
L_\bullet: \ \ \cdots \arr L_2 \arr L_1 \arr L_0\arr 0.
\]
The $n$-th homology of $G$ with coefficient in $L_\bullet$, denoted by 
\[
H_n(G, L_\bullet),
\]
is defined as the $n$-th homology of the total complex of the double complex 
$\mathcal{C}_\bullet(G) \otimes_{G} L_\bullet$, where $\mathcal{C}_\bullet(G)\arr \z$ 
is the standard resolution of $G$. This double complex 
%\mathcal{C}_\bullet(G) \otimes_{G} L_\bullet$, 
induces two spectral sequences
\[
\mathsf{E}_{p,q}^2(G)=H_p(G, H_q(L_\bullet))\Rightarrow H_{p+q}(G, L_\bullet),
\]
and 
\[
E_{p, q}^1(G)=H_q(G, L_p) \Rightarrow H_{p+q}(G, L_\bullet),
\]
(see \cite[\S 5, Chap. VII]{brown1994}).

\begin{lem}\label{homology-complex}
Let the complex $L_\bullet$ be exact for $1\leq  i \leq n$ and $M=H_0(L_\bullet)$. Then 
$H_i(G, L_\bullet)\simeq H_i(G, M)$ for $0 \leq  i \leq n$. 
\end{lem}
\begin{proof}
This follows from an easy analysis of the spectral sequence $\mathsf{E}_{p,q}^2(G)$.
%\[
%\mathsf{E}_{p,q}^2(G)=H_p(G, H_q(L_\bullet))\Rightarrow H_{p+q}(G, L_\bullet).
%\]
\end{proof}

From the complex $C_\bullet(R^2)$ 
of $\GL_2(R)$-modules we obtain the first quadrant spectral sequence
\[
E_{p, q}^1:=E_{p, q}^1(\GL_2(R))=H_q(\GL_2(R), C_p(R^2)) 
\Rightarrow H_{p+q}(\GL_2(R), C_\bullet(R^2)).
\]
In this section we will study this spectral sequence for small values of $p$ and $q$. 

\begin{lem}[Hutchinson]\label{hutchinson1}
The complex $C_\bullet(R^2)$ is exact for $1\! \leq i <\! |R/\mmm_R|$ and
$H_0(C_\bullet(R^2))\simeq \z$.
% and thus
%. Thus for $0 \leq i < |R/\mmm_R|$,
%\[
%$H_i(\GL_2(R), C_\bullet(R^2))\simeq H_i(\GL_2(R))$.
%\]
\end{lem}
\begin{proof}
See \cite[Lemma 3.21]{hutchinson2014}. 
\end{proof}

For simplicity, in the rest of this section we will assume
\[
{\bf \infty}:=\lan e_1 \ran, \ \  {\bf 0}:=\lan e_2 \ran, \ \ {\bf 1}:=\lan e_1+e_2 \ran, \ \
{\bf b^{-1}}:=\lan e_1+be_2 \rangle, \ b \in \rr. 
\]
Since
$\GL_2(R)$ acts transitively on the sets of generators of $C_i(R^2)$ for $i=0,1,2$,
by the Shapiro lemma we have
\begin{align*}
& E_{0,q}^1 \simeq H_q(\stabe_{\GL_2(R)}(\infty))=H_q(B_2),\\
& E_{1,q}^1 \simeq H_q(\stabe_{\GL_2(R)}({ \infty, {\bf 0}}))=H_q(T_2),\\
& E_{2,q}^1 \simeq H_q(\stabe_{\GL_2(R)}({ \infty, {\bf 0},{\bf 1}}))=H_q(\rr),
\end{align*}
where $B_2:={\mtx \rr R 0 \rr}$ and $T_2:={\mtx \rr 0 0 \rr}\simeq \rr \times \rr$. 
Moreover, the orbits of the action of $\GL_2(R)$ on $C_3(R^2)$ and $C_4(R^2)$ 
are given by the frames  $p(a):=({ {\bf \infty, 0, 1, a^{-1}}})$, and 
$q(a,b):=({ {\bf \infty, 0, 1, a^{-1}, b^{-1}}})$, respectively, where 
$a, 1-a, b, 1-b, a-b \in \rr$.
Therefore
\begin{align*}
& 
\begin{array}{c}
E_{3,q}^1\simeq \bigoplus_{a \in \Aa} H_q(\rr).p(a),
%= H_q(\rr)\otimes\bigg(\bigoplus_{a \in \Aa} \z.p(a)\bigg),
\end{array}\\
&
\begin{array}{c} 
E_{4,q}^1\simeq\bigoplus_{(a, b) \in \Bb} H_q(\rr).q(a, b),
%=H_q(\rr)\otimes \bigg(\bigoplus_{(a, b) \in \Bb} \z.q(a, b)\bigg),
\end{array}
\end{align*}
where $\Aa:=\{a \in R: a, 1-a \in \rr\}$ and 
$\Bb:=\{(a,b)\in R^2: a, 1-a, b, 1-b, a-b \in \rr\}$.

If $N_2:=\mtx 1 R 0 1 \se B_2={\mtx \rr R 0 \rr}$,
then 
%clearly $N_2\simeq R$ and
%\[
$B_2=N_2 \rtimes T_2\simeq R \rtimes T_2$,
%\]
where the action of $T_2$ on $R$ is given by $(a,b).x:=ab^{-1}x$.
This clearly implies that for any $n \geq 0$, 
\[
H_n(B_2)\simeq H_n(T_2) \oplus A_n,
\]
where $A_n$ is the kernel of  $H_n(B_2) \arr H_n(T_2)$ which is induced by the
natural surjective map $B_2 \arr T_2$.

\begin{lem}\label{T2-B2}
%Let the residue field of $R$ has at least tree elements. Then
If $|R/\mmm_R|\neq 2$, then
$H_1(B_2)\simeq H_1(T_2)$, $H_2(B_2)\simeq H_2(T_2) \oplus H_2(R)_\rr$ and
$H_3(B_2)\simeq H_3(T_2) \oplus A_3$, where $A_3$ sits in the exact sequence
\[
H_2(T_2, H_2(R)) \arr H_3(R)_\rr \arr A_3 \arr H_1(T_2, H_2(R)) \arr 0.
\]
Here the action of $\rr$ on $H_i(R)$ is induced by the natural action of $\rr$
on $R$, i.e. $a.r:=ar$.
\end{lem}
\begin{proof}
From the extension  $0 \arr R \arr B_2 \arr T_2\arr 1$,
we obtain the Lyndon-Hochschild-Serre spectral sequence 
\[
\EE_{r,s}^2=H_r(T_2,H_s(R))\Rightarrow H_{r+s}(B_2).
\]
Since $R/\mmm_R$ has at least three elements, 
there is $a \in \rr$ such that $a-1 \in \rr$. Then 
\[
H_0(T_2,H_1(R))=H_0(T_2,R)=R/\lan a-1:a \in \rr\ran=0,
\]
and so by Lemma~\ref{lem-sus} below (for $\varphi: T_2\arr \rr$, $(a,b)\mt ab^{-1}$), 
for any $r\ge 0$, we have $\EE_{r,1}^2=H_r(T_2,H_1(R))=0$.
Now by an easy analysis of the above spectral sequence we 
obtain the desired results.
\end{proof}

\begin{lem}[Suslin]\label{lem-sus}
Let $G$ be an abelian group, $A$ a commutative ring, $M$ an $A$-module and 
$\varphi: G \arr A^\times$ a homomorphism of groups which turns $A$ and $M$ 
into $G$-modules.  If $H_0(G,A)=0$, then for any
$n\geq 0$, $H_n(G,M)=0$.
\end{lem}
\begin{proof}
See \cite[Lemma~1.8]{nes-suslin1990}.
\end{proof}

The pre-Bloch group $\ppp(R)$ of a commutative ring $R$ is the quotient of
the free abelian group $Q(R)$ generated by symbols $[a]$, $a, 1-a \in \rr$,
by the subgroup generated by elements of the form
\[
[a] -[b]+\bigg[\frac{b}{a}\bigg]-\bigg[\frac{1- a^{-1}}{1- b^{-1}}\bigg]
+ \bigg[\frac{1-a}{1-b}\bigg],
\]
where $a, 1-a, b, 1-b, a-b  \in \rr$. Let the map
$\lambda: Q(R) \arr \rr \otimes \rr$ be defined by $[a] \mapsto a \otimes (1-a)$.
Then by a direct computation we have
\[
\lambda\Big(
[a] -[b]+\bigg[\frac{b}{a}\bigg]-\bigg[\frac{1- a^{-1}}{1- b^{-1}}\bigg]
+ \bigg[\frac{1-a}{1-b}\bigg] \Big)
=a \otimes \bigg( \frac{1-a}{1-b}\bigg)+\bigg(\frac{1-a}{1-b}\bigg)\otimes a.
\]
Let 
\[
(\rr \otimes \rr)_\sigma :=\rr \otimes \rr/
\lan c\otimes d + c\otimes d: c, d \in \rr \ran.
\]
We denote the elements of $\ppp(R)$ and $(\rr \otimes \rr)_\sigma$
represented by $[a]$ and $a\otimes b$ again by $[a]$ and $a\otimes b$,
respectively. Thus we have the well-defined map, denoted again by $\lambda$,
\[
\lambda: \ppp(R) \arr (\rr \otimes \rr)_\sigma, \ \ \
[a] \mapsto a \otimes (1-a).
\]
The kernel of $\lambda$ is called the Bloch group of $R$ and is
denoted by $B(R)$. Thus we obtain the exact sequence
\[
0 \arr B(R) \arr \ppp(R) \arr (\rr \otimes \rr)_\sigma \arr K_2^{MS}(R) \arr 0,
\]
where
\[
K_2^{MS}(R):= \rr \otimes \rr/
\lan a \otimes(1-a), b\otimes c + c\otimes b: a,1-a, b, c \in \rr \ran.
\]

The $n$-th Milnor $K$-group of a commutative ring $R$
is defined as the abelian group $K_n^M(R)$ generated by symbols
$\{a_1, \dots, a_n\}$, $a_i \in \rr$,  $i=1, \dots, n$,
subject to the following relations
\medskip
\par (i) $\{a_1, \dots,a_ia_i', \dots, a_n\!\}\!= 
\!\{a_1, \dots,a_i, \dots, a_n\!\}\!+\!
\{a_1, \dots,a_i', \dots, a_n\!\}$, any $i$,
\medskip
\par (ii) $\{a_1,  \dots, a_n\}=0$ if there exist $i,j$, $ i \neq j$,
such  that $a_i+a_j=0$ or $1$.
\medskip
~\\
Clearly we have the anti-commutative product map
\[
K_m^M(R) \otimes_\z K_n^M(R) \arr K_{m+n}^M(R),
\]
\[ 
\{a_1, \dots,a_m\}\otimes \{b_1, \dots, b_n\} \mt \{a_1, \dots,a_m, b_1, \dots, b_m\}.
\] 
Since in $\rr \otimes \rr$ we have
$b\otimes c + c\otimes b=bc \otimes(-bc)- b\otimes (-b) - c\otimes (-c)$,
the natural map  $K_2^{MS}(R) \arr K_2^{M}(R)$ is surjective. 

\begin{lem}\label{vdk00}
%Let $R$ be a local ring. 
\par {\rm (i)} Let $R$ be either a field or a local ring with $|R/\mmm_R|\neq 2$.
Then $K_2^{M}(R)\simeq K_2^{MS}(R)$. 
\par {\rm (ii)} 
Let $R$ be either a field or a local ring with $|R/\mmm_R|>5$. Then 
\[
K_2^{M}(R)\simeq \rr \otimes \rr/
\lan a \otimes(1-a): a,1-a \in \rr \ran.
\]
In particular, $K_n^M(R)\simeq \Big(\bigotimes_\z^n \rr \Big)/T$, where $T$ is 
the subgroup generated by the elements $a_1\otimes \cdots \otimes a_n$ such 
that there are $i,j$, $i\neq j$, with $a_i+a_j=1$.
\end{lem}
\begin{proof}
(i) We denote the element of $K_2^{MS}(R)$ represented by $a\otimes b$, by 
$\{a,b\}_{MS}$. Thus it is sufficient to prove that for any $a\in \rr$, $\{a,-a\}_{MS}=0$.
If $1-a \in \rr$, then $-a=(1-a)/(1-a^{-1})$ and so 
\begin{align*}
\{a,-a\}_{MS} =\{a,1-a\}_{MS}-\{a,1-a^{-1}\}_{MS}=\{a^{-1},1-a^{-1}\}_{MS}=0.
\end{align*}
This covers the case of fields. Let $R$ be local and $1-a \notin \rr$. Since 
$|R/\mmm_R|\neq 2$, there is always $c \in \rr$ such that $c, 1-c, 1-ac\in \rr$. 
Then by the above argument $\{c,-c\}_{MS}=\{ac,-ac\}_{MS}=0$ and thus 
\begin{align*}
\{a,-a\}_{MS}&= \{ac,-ac\}_{MS}-\{c,-c\}_{MS}-\{a,c\}_{MS}-\{c,a\}_{MS}=0.
\end{align*}
(ii) This can be done as Steps 1-5 in the proof of \cite[Theorem~8.4]{vdkallen1977}.
\end{proof}

%Thus if $|R/\mmm_R|\neq 2$, then we have the exact sequence 
%\[
%0 \arr B(R) \arr \ppp(R) \arr (\rr \otimes \rr)_\sigma \arr K_2^M(R) \arr 0.
%\]

Now let $|R/\mmm_R|\neq 2$. Then by Lemmas \ref{hutchinson1} and \ref{homology-complex}, 
for $n=0,1,2$ we have
\[
H_{n}(\GL_2(R), C_\bullet(R^2))\simeq H_{n}(\GL_2(R)).
\]
If $|R/\mmm_R|\neq 2, 3$, then we also have the above isomorphism for $n=3$.

%In this section we will study the differentials of the spectral sequence $E_{p,q}^1$ for
%small values of $p$ and $q$. 
Now we study the differentials of the spectral sequence $E_{p,q}^1$ for
small values of $p$ and $q$. It is not difficult to see that 
\[
d_{1, q}^1=H_q(\sigma) - H_q(\sigma'),
\]
where $\sigma, \sigma': T_2 \arr B_2$ are given by $\sigma(a, b)= \diag(b, a)$ and 
$\sigma'(a, b)= \diag(a, b)$.
Moreover,
\[
d_{2,q}^1=H_q(\Delta),
\]
where $\Delta: \rr \arr T_2$ is the diagonal map $a \mt (a,a)$. Hence $d_{2,q}^1$ 
always is injective. On the other hand, by a direct computation one can show that 
for any $z \in H_q(\rr)$,
\[
d_{3,q}^1\Big(z.p(a)\Big)=0
\]
and
\[
d_{4,q}^1\Big(z.q(a,b)\Big)= z.\bigg(p(a)-p(b)+p\bigg(\frac{b}{a}\bigg)
-p\bigg(\frac{1- a^{-1}}{1- b^{-1}}\bigg)
+ p\bigg(\frac{1-a}{1-b}\bigg)\bigg).
\]
Putting all these together, the 
${E}^2$-terms of our spectral sequence look as follow:
\begin{gather*}
\begin{array}{cccccc}
\ast        &          &          &        &  & \\
&&&&&\\  
H_3(T_2)_\sigma \oplus A_3 &   H_3(T_2)^\sigma/H_3(\rr)   &  0       &        &  &\\ 
&&&&&\\  
H_2(T_2)_\sigma\oplus H_2(R)_\rr &(\rr \otimes \rr)^\sigma &0 &H_2(\rr)\otimes \ppp(R)& &\\
&&&&&\\  
\rr         &   0      &  0       &  \rr \otimes \ppp(R)     & \ast & \\
&&&&&\\  
\z          &   0      &  0       &  \ppp(R)  & \ast & 
\end{array}
\end{gather*}
Note that $H_2(T_2)_\sigma\simeq  H_2(\rr) \oplus (\rr \otimes \rr)_\sigma$.

For an arbitrary group $G$, let $\mathcal{C}_\bullet(G)\overset{\epsilon}{\arr}\z$ 
and $\mathcal{B}_\bullet(G)\overset{\varepsilon^G}{\arr}\z$ denote the (left) standard and 
the (left) bar resolution of $G$, respectively.
We turn $\mathcal{C}_n(G)$ and $\mathcal{B}_n(G)$ into a right $G$-module in usual way.
Note that the map
\[
\mathcal{C}_n(G) \arr \mathcal{B}_n(G), \ \
(g_0, \dots, g_n) \mt
[g_0g_1^{-1}|g_1g_2^{-1}|\dots|g_{n-2}g_{n-1}^{-1}|g_{n-1}g_{n}^{-1}]
\]
induces the identity map of $H_n(G)$. For simplicity the element of $H_n(G)$ represented by 
$\sum m[g_1|\dots|g_n]$ again is denoted by $\sum m[g_1|\dots|g_n]$.
 
\begin{lem}\label{G-H}
Let $H$ be a subgroup of $G$. Let $\theta:G/H \arr G$ be any (set theoretic)
section of the natural map (of sets) $\pi:G \arr G/H$, $g \mt gH$. For $g \in G$, let 
$\overline{g}:=\Big(\theta\circ\pi(g)\Big)^{-1}g$. Then the map 
$\mathcal{C}_n(G) \arr \mathcal{C}_n(H)$
given by $(g_0, \dots, g_n) \mt (\overline{g_0}, \dots, \overline{g_n})$ induces an $H$-morphism 
of the standard complexes $\mathcal{C}_\bullet(G) \arr \mathcal{C}_\bullet(H)$ and for any $n$, 
the homomorphism 
\[
H_n(H)= H_n(\mathcal{C}_\bullet(G)_H) \arr H_n(\mathcal{C}_\bullet(H)_H)= H_n(H)
\]
coincides with the
identity map $\id_{H_n(H)}$.
\end{lem}
\begin{proof}
This is easy to prove. In fact this map
can be seen as the inverse of the identity homomorphism
\[
H_n(H)=H_n(\mathcal{C}_\bullet(H)_H) \arr H_n(\mathcal{C}_\bullet(G)_H)= H_n(H)
\]
induced by the inclusion of $H$ in $G$.
\end{proof}
 
For any $n$-tuple $(g_1, g_2, \dots, g_n)$ of 
pairwise commuting elements of G, let
\[
{\rm \bf{c}}({g}_1, {g}_2,\dots, {g}_n):=\sum_{\sigma \in \Sigma_n} {{\rm
sign}(\sigma)}[{g}_{\sigma(1)}| {g}_{\sigma(2)}|\dots|{g}_{\sigma(n)}] \in
H_n(G),
\]
where $\Sigma_n$ is the symmetric group of degree $n$. 
In fact, ${\rm \bf{c}}({g}_1, {g}_2,\dots, {g}_n)$ is
the image of $g_1\wedge \dots \wedge g_n$ under the composition
\[
\begin{array}{c}
\bigwedge_\z^n A \arr H_n(A) \arr H_n(G),
\end{array}
\]
where $A$ is the abelian subgroup of $G$ generated by $g_1,\dots, g_n$
and the first map is the Pontryagin product.

\begin{lem}\label{d3}
The differential map
\[
d_{3,0}^3 :\ppp(R) \arr H_2(\rr) \oplus (\rr \otimes \rr)_\sigma \oplus H_2(R)_\rr
\]
is given by 
\[
d_{3,0}^3([a])= \bigg(a \wedge (1-a),-a\otimes (1-a),-2{\bf c}(a, 1-a) \bigg).
\]
\end{lem}
\begin{proof}
This is a long and tedious calculation.  
Here we argue as in \cite{mirzaii2012}. 
For simplicity let 
$F_i:=\mathcal{C}_i(\GL_2(R))$, $C_i:=C_i(R^2)$, $\GL_2=\GL_2(R)$ and 
consider the following commutative diagram
\[
\begin{array}{cccccccc}
F_2 \otimes_{\GL_2} C_3 \!\! &\!\! \larr \!\! &\!\! F_2 
\otimes_{\GL_2} C_2 \!\!&\!\! \larr\!\!
& \!\! F_2 \otimes_{\GL_2} C_1 \!\! &  \!\!
{\larr} \!\! &\!\! \!\!F_2 \otimes_{\GL_2} C_0 \\
\Bigg\downarrow\vcenter{%
\rlap{$\scriptstyle{}$}}&  &\Bigg\downarrow\vcenter{%
\rlap{$\scriptstyle{}$}}
&       & \Bigg\downarrow\vcenter{%
\rlap{}}
&  &\Bigg\downarrow{%
\rlap{$\scriptstyle{}$}} \\
F_1 \otimes_{\GL_2} C_3 \!\! &\!\! \larr \!\! &\!\! F_1 
\otimes_{\GL_2} C_2 \!\!&\!\! \larr\!\!
& \!\! F_1 \otimes_{\GL_2} C_1 \!\! &  \!\!
{\larr} \!\! & \!\!\!\!F_1\otimes_{\GL_2} C_0 \\
\Bigg\downarrow\vcenter{%
\rlap{$\scriptstyle{}$}}&  &\Bigg\downarrow\vcenter{%
\rlap{$\scriptstyle{}$}}
&       & \Bigg\downarrow\vcenter{%
\rlap{}}
&  &\Bigg\downarrow{%
\rlap{$\scriptstyle{}$}} \\
F_0 \otimes_{\GL_2} C_3 \!\! &\!\!\larr \!\! &\!\! F_0 \otimes_{\GL_2} C_2
 \!\!&\!\! {\larr}\!\!
& \!\! F_0 \otimes_{\GL_2} C_1 \!\! & \!\! \larr \!\!
& \!\!\!\! F_0 \otimes_{\GL_2}\! C_0.
\end{array}
\]
The element $[a] \in \ppp(R)$ comes from
$x_a:=(1) \otimes ({\bf \infty, 0, 1, a^{-1}}) \in F_0 \otimes_{\GL_2} C_3$,
and $(1) \otimes \partial_3({\bf  \infty, 0, 1, a^{-1}})\!= \!
\bigg[(g_1)-(g_2)+(g_3)-(1)\bigg] \otimes ({\bf\infty, 0, 1})
\!\in\! F_0\otimes_{\GL_2} C_2$,
where
\[
g_1=\mtx 0 1 {a-1} 1, \ \ g_2=\mtx  {1-a} a 0 a ,\ \ g_3= \mtx 1 0 0 a.
\]
If
$y_a:=  \bigg[(g_2, g_1)-(g_3, 1)\bigg] \otimes ({\bf\infty, 0, 1})\in F_1\otimes_{\GL_2} C_2$,
then $\delta_1\otimes \id_{C_2}(y_a)=x_a$. Now
$z_a:=(\id_{F_1}\otimes \partial_2)(y_a)=
\bigg[(g_2, g_1)-(g_3, 1)\bigg] \otimes \partial_2({\bf\infty, 0, 1})
\in F_1 \otimes_{\GL_2} C_1$,
is equal to
\[
 \bigg[(g_2 g_1,g_1^2)-(g_3 g_1,g_1)-(g_2^2, g_1 g_2)
+(g_3 g_2, g_2)+(g_2, g_1)-(g_3, 1)\bigg] \otimes ({\bf\infty, 0}).
\]
Since $a^{-1}g_2 g_1=g_1^2 g_3^{-1}$,
$(a^{-1} g_2^2 g_3^2, g_1 g_2)=  (g_3g_2,g_3g_1){\mtx {a^{-1}(1-a)} 0 0 a }$,
by a direct calculation one can see that $ \delta_2 (u_a)\otimes ({\bf\infty, 0})=z_a$,
where
\begin{align*}
u_a =& +(g_3 g_1, g_2, g_1)-(g_3 g_2, g_3 g_1, g_2)
-(a^{-1} g_2^2 g_3^2, g_2^2, g_1 g_2)\\
& +(a^{-3} g_2^2 g_3^2, a^{-2}g_2^2, 1)
-(a^{-3} g_2^2 g_3^2, a^{-1}g_3^2, 1)
+ (a^{-1}g_2 g_1, a^{-1}g_1^2, 1)\\
& -(g_1^2g_3^{-1}, ag_3^{-1}, 1)
+ (g_3, a^{-1}g_3^2, 1).
\end{align*}
Here by $ag$, $a \in \rr$ and $g \in \GL_2$, we mean $\mtx a {0} {0} a g$. Note that 
$d_{3,0}^2([a]) \in E_{0,2}^2 = H_2(B_2)/\im(d_{1,2}^1)$ is
represented by
\[
u_a \otimes \partial_1({\bf\infty, 0}) =(u_aw-u_a)\otimes (\infty) \in
F_2 \otimes_{\GL_2} C_0,
\]
where $w=\mtx 0 1 1 0$. On the chain level the isomorphism
\begin{equation}\label{h2b}
H_2(\GL_2(R), C_0(R^2))\overset{\simeq}{\larr} H_2(B_2)
\end{equation}
is given by
$F_\bullet \otimes_{\GL_2} C_0 \arr F_\bullet \otimes_{B_2} \z$, 
$y \otimes (\infty) \mt y \otimes 1$.
Let $\mathcal{C}_\bullet(B_2)\arr \z$ be the standard resolution of $\z$ over $B_2$.
By Lemma \ref{G-H}, an augmented preserving chain map of $B_2$-resolutions
\[
F_\bullet=\mathcal{C}_\bullet(\GL_2(R)) \arr \mathcal{C}_\bullet(B_2)
\]
is obtained as follows: The map $s:\GL_2/B_2 \arr \GL_2$ given by
\[
s(gB):=\begin{cases}
1 & \text{if $g(\infty)=\infty$}\\
w & \text{if $g(\infty)={\bf 0}$ }\\
{\mtx 1 0 {b} 1} & \text{if $g(\infty)={\bf b^{-1}}$,}
\end{cases}
\]
is a (set-theoretic) section of the canonical projection $\pi: \GL_2 \arr \GL_2/B_2$. 
Now if $\overline{g}:=(s\circ\pi(g))^{-1}g$, then we have
\[
\overline{g}=\begin{cases}
g & \text{if $g(\infty)=\infty$}\\
wg & \text{if $g(\infty)={\bf 0}$ }\\
{\mtx 1 0 {-b} 1}g & \text{if $g(\infty)={\bf b^{-1}}$.}
\end{cases}
\]
Thus on the chain level the map
\[
F_n \otimes_{\GL_2} C_0 \arr \mathcal{C}_n(B_2) \otimes_{B_2} \z, \ \
(g_0, \dots, g_n) \otimes (\infty)  \mt
(\overline{g_0}, \dots, \overline{g_n}) \otimes 1,
\]
induces the homomorphism (\ref{h2b}). 

Hence by a direct computation we see that under the map (\ref{h2b}) the element 
$d_{3,0}^2([a]) \in H_2(B_2)=H_2(\mathcal{B}_\bullet(B_2)_{B_2})$
is represented by the element $X_a \in \mathcal{B}_2(B_2)_{B_2}$, where
\begin{align*}
X_a\!\!=&+\bigg[{\mtxx {a^{-1}} {a^{-1}} {0} {a}}| {\mtxx {a} {-1} {0} {1}} \bigg]
 -\bigg[{\mtxx {-a} {a+1} {0} {a^{-1}}}| {\mtxx {-1} {a+1} {0} {a}}   \bigg]\\
&\!\!\!\!\!\!\!\!\!-\!\!
\bigg[{\mtxx {a} {-a^{-1}} {0} {1}}| {\mtxx {a^{-1}} {a^{-1}} {0} {a}} \bigg]
+\bigg[{\mtxx {-a^{-1}} {a+1} {0} {a^2}}| {\mtxx {-a} {a+1} {0} {a^{-1}}} \bigg]\\
&\!\!\!\!\!\!\!\!\!-\!\!
\bigg[{\mtxx {a} {1-a^{-2}} {0} {a^{-1}}}| {\mtxx {1} {-1} {0} {a}}\bigg]
+\bigg[{\mtxx {a^{-1}} {1-a^{-2}} {0} {a}}| {\mtxx {-1} {a+1} {0} {a}}\bigg]\\
%%%%%%
&\!\!\!\!\!\!\!\!\!+\!\!
\bigg[\!{\mtxx {a}{1-a^{-2}}{0}{a^{-1}}}\!|
\!{\mtxx {a^{-1}}{a^{-2}(a-1)^{2}}{0}{-a^{-1}(a-1)^{2}}}\!\bigg]
\!\!-\!\!\bigg[\!{\mtxx {a^{-1}} {1-a^{-2}} {0} {a}}\!|
\!{\mtxx {a^{-2}(a-1)^2} {a^{-1}} {0} {1}}\!\bigg]\\
&\!\!\!\!\!\!\!\!\!-\!\!
\bigg[{\mtxx {a^{-1}} {a^{-2}(a-1)^2} {0} {-a^{-1}(a-1)^2}}| {\mtxx{a} {0} {0}{a^{-1}}}\bigg]
+\bigg[{\mtxx {a^{-2}(a-1)^2} {a^{-1}} {0} {1}}| {\mtxx {a^{-1}} {0} {0} {a}}\bigg]\\
&\!\!\!\!\!\!\!\!\!+\!\!
\bigg[{\mtxx{1}{-1}{0} {a}}|{\mtxx {a^{-1}} {a^{-1}(a-1)} {0} {-a^{-1}(a-1)^2}}\bigg]
\!-\! \bigg[{\mtxx {a} {-1} {0} {1}}| {\mtxx {a^{-1}(a-1)} {a^{-1}} {0} {a^{-1}(a-1)}}\bigg]\\
&\!\!\!\!\!\!\!\!\!-\!\!
\bigg[{\mtxx {a^{-1}}{a^{-1}(a-1)} {0} {-a^{-1}(a-1)^2}}|{\mtxx {1} {0} {0} {a}}\bigg]
+\bigg[{\mtxx {a^{-1}(a-1)} {a^{-1}} {0} {a^{-1}(a-1)}}| {\mtxx {a} {0} {0} {1}}\bigg]\\
&\!\!\!\!\!\!\!\!\!+\!\!
\bigg[{\mtxx {1} {0} {0} {a}}| {\mtxx {a} {0} {0} {a^{-1}}}\bigg]
-\bigg[{\mtxx {a} {0} {0} {1}}| {\mtxx {a^{-1}} {0} {0} {a}}\bigg].
\end{align*}
Let $Y_a\in H_2(T_2)_\sigma$ and $Z_a\in H_2(N_2)_\rr$ be
\begin{align*}
Y_a= 
& +[(a^{-1}, a)|(a, 1)]- [(-a,a^{-1})|(-1, a)] -[(a, 1)|(a^{-1}, a)]\\
& +[(-a^{-1}, a^2)|(-a, a^{-1})]\!-\![(a, a^{-1})|(1, a)]\!+\![(a^{-1}, a)|(-1, a)] \\
& +[(a, a^{-1})|(a^{-1},-a^{-1}(a-1)^2)]-[(a^{-1},a)|(a^{-2}(a-1)^2, 1)]\\
& -[(a^{-1}, -a^{-1}(a-1)^2)|(a, a^{-1})]+[(a^{-2}(a-1)^2, 1)|(a^{-1}, a)]\\
& +[(1,a)|(a^{-1}, -a^{-1}(a-1)^2))]\!-\![(a,1)|(a^{-1}(a-1),\!a^{-1}(a-1))]\\
& -[(a^{-1}, -a^{-1}(a-1)^2)|(1,a)]\!+\![(a^{-1}(a-1),\! a^{-1}(a-1))|(a, 1)]\\
& +[(1, a)|(a,a^{-1})]-[(a,1)|(a^{-1},a)],
\end{align*}
\begin{align*}
Z_a=&+\bigg[{\mtxx {1} {-1} {0} {1}}| {\mtxx {1} {1} {0} {1}} \bigg] -
\bigg[{\mtxx {1} {1} {0} {1}}| {\mtxx {1} {-1} {0} {1}} \bigg]\\
&+ \bigg[{\mtxx {1} {-a^{-2}} {0} {1}}| {\mtxx {1} {a^{-2}} {0} {1}} \bigg] -
\bigg[{\mtxx {1} {a^{-2}} {0} {1}}| {\mtxx {1} {-a^{-2}} {0} {1}} \bigg]\\
& + \bigg[{\mtxx {1} {a(a+1)} {0} {1}}| {\mtxx {1} {-a(a+1)} {0} {1}} \bigg] -
\bigg[{\mtxx {1} {-a(a+1)} {0} {1}}| {\mtxx {1} {a(a+1)} {0} {1}} \bigg]\\
& +\bigg[{\mtxx {1} {-1} {0} {1}}| {\mtxx {1} {a(a-1)^{-1}} {0} {1}} \bigg] -
\bigg[{\mtxx {1} {a(a-1)^{-1}} {0} {1}}| {\mtxx {1} {-1} {0} {1}} \bigg]\\
& +\bigg[{\mtxx {1} {a+1} {0} {1}}| {\mtxx {1} {-1} {0} {1}} \bigg] -
\bigg[{\mtxx {1} {-1} {0} {1}}| {\mtxx {1} {a+1} {0} {1}} \bigg]\\
& +\bigg[{\mtxx {a} {0} {0} {a}}| {\mtxx {1} {1} {0} {1}} \bigg] -
\bigg[{\mtxx {1} {1} {0} {1}}| {\mtxx {a} {0} {0} {a}} \bigg]\\
& +\bigg[{\mtxx {1} {-a(a-1)^{-1}} {0} {1}}| {\mtxx {a} {0} {0} {a}} \bigg] -
\bigg[{\mtxx {a} {0} {0} {a}}| {\mtxx {1} {-a(a-1)^{-1}} {0} {1}} \bigg].
\end{align*}
By a direct computation, one sees that in $\mathcal{B}_2(B_2)_{B_2}$ we have
\[
X_a=Y_a-Z_a +\delta_3(W_a),
\]
where $W_a$ is the following element of 
$\mathcal{B}_3(B_2)_{B_2}\simeq \z\otimes_{B_2}\mathcal{B}_3(B_2)$:
\begin{align*}
W_a\!=& \!-\!\!\bigg[{\mtxx {a^{-1}} {0} {0} {a}}| 
{\mtxx {1} {1} {0} {1}}| {\mtxx {a} {-1} {0} {1}} \bigg]
\!\!-\!\!\bigg[{\mtxx {1} {-a^{-1}(a\!+\!1)} {0} {1}}| 
{\mtxx {1} {a^{-1}(a\!+\!1)} {0} {1}}| 
{\mtxx {-1} {0} {0} {a}} \bigg]\\
& +\bigg[{\mtxx {-a} {0} {0} {a^{-1}}}| 
{\mtxx {1} {-a^{-1}(a+1)} {0} {1}}| 
{\mtxx {-1} {a+1} {0} {a}} \bigg]
+\bigg[{\mtxx {1} {1} {0} {1}}| 
{\mtxx {1} {-1} {0} {1}}| 
{\mtxx {a} {0} {0} {1}} \bigg]\\
& +\bigg[{\mtxx {a} {0} {0} {1}}| 
{\mtxx {1} {-a^{-2}} {0} {1}}| 
{\mtxx {a^{-1}} {a^{-1}} {0} {a}} \bigg]
-\bigg[{\mtxx {1} {-a^{-2}} {0} {1}}| 
{\mtxx {1} {a^{-2}} {0} {1}}| 
{\mtxx {a^{-1}} {0} {0} {a}} \bigg]\\
& -\bigg[{\mtxx {-a^{-1}} {0} {0} {a^{2}}}| 
{\mtxx {1} {-a(a+1)} {0} {1}}| 
{\mtxx {-a} {a+1} {0} {a^{-1}}} \bigg]
-\bigg[{\mtxx {1} {1} {0} {1}}| 
{\mtxx {a} {0} {0} {1}}| 
{\mtxx {1} {0} {0} {a}} \bigg]\\
&+\bigg[{\mtxx {1} {-a(a+1)} {0} {1}}| 
{\mtxx {1} {a(a+1)} {0} {1}}| 
{\mtxx {-a} {0} {0} {a^{-1}}} \bigg]
+ \bigg[{\mtxx {a} {0} {0} {1}}| 
{\mtxx {1} {a^{-1}} {0} {1}}| 
{\mtxx {1} {0} {0} {a}} \bigg]\\
& +\bigg[{\mtxx {a} {0} {0} {a^{-1}}}| 
{\mtxx {1} {a^{-1}(1-a^{-2})} {0} {1}}| 
{\mtxx {1} {-1} {0} {a}} \bigg]
-\bigg[{\mtxx {a^{-1}} {0} {0} {a}}| 
{\mtxx {a} {0} {0} {1}}| 
{\mtxx {1} {1} {0} {1}} \bigg]\\
&-\bigg[{\mtxx {1} {a^{-1}(1-a^{-2})} {0} {1}}| 
{\mtxx {1} {-a^{-1}} {0} {1}}| 
{\mtxx {1} {0} {0} {a}} \bigg]
-\bigg[{\mtxx {1} {a^{-1}} {0} {1}}| 
{\mtxx {a^{-1}} {0} {0} {a}}| 
{\mtxx {a} {0} {0} {1}} \bigg]\\
& -\bigg[{\mtxx {a} {0} {0} {a^{-1}}}| 
{\mtxx {1} {0} {0} {a}}| 
{\mtxx {1} {-a^{-2}} {0} {1}} \bigg]
-\bigg[{\mtxx {a^{-1}} {0} {0} {a}}| 
{\mtxx {1} {a(1-a^{-2})} {0} {1}}| 
{\mtxx {-1} {a+1} {0} {a}} \bigg]\\
\end{align*}
\begin{align*}
& +\bigg[{\mtxx {1} {a(1-a^{-2})} {0} {1}}| 
{\mtxx {1} {a^{-1}(a+1)} {0} {1}}| 
{\mtxx {-1} {0} {0} {a}} \bigg]
+\bigg[{\mtxx {a^{-1}} {0} {0} {a}}| 
{\mtxx {1} {a} {0} {1}}| 
{\mtxx {a} {0} {0} {1}} \bigg]\\
& +\bigg[{\mtxx {a} {1-a^{-2}} {0} {a^{-1}}}| 
{\mtxx {1} {-a^{-1}} {0} {1}}| 
{\mtxx {a^{-1}} {0} {0} {-a^{-1}(a-1)^2}}\bigg]
+\bigg[{\mtxx {1} {a} {0} {1}}| 
{\mtxx {a} {0} {0} {1}}| 
{\mtxx {1} {-1} {0} {1}} \bigg]\\
&-\bigg[{\mtxx {1} {-a^{-1}} {0} {1}}| 
{\mtxx {a} {0} {0} {a^{-1}}}| 
{\mtxx {a^{-1}} {0} {0} {-a^{-1}(a-1)^2}} \bigg]
-\bigg[{\mtxx {1} {-1} {0} {1}}| 
{\mtxx {1} {1} {0} {1}}| 
{\mtxx {a} {0} {0} {1}} \bigg]\\
& -\bigg[{\mtxx {a} {0} {0} {a^{-1}}}| 
{\mtxx {1} {a^{-1}(1-a^{-2})} {0} {1}}| 
{\mtxx {1} {-a^{-1}} {0} {1}} \bigg]
+\bigg[{\mtxx {1} {a^{-2}} {0} {1}}| 
{\mtxx {1} {-a^{-2}} {0} {1}}| 
{\mtxx {a^{-1}} {0} {0} {a}} \bigg]\\
&-\bigg[{\mtxx {a^{-1}} {1-a^{-2}} {0} {a}}| 
{\mtxx {1} {a^{-1}} {0} {1}}| 
{\mtxx {a^{-2}(a-1)^2} {0} {0} {1}} \bigg]
- \bigg[{\mtxx {1} {a^{-2}} {0} {1}}| 
{\mtxx {a^{-1}} {0} {0} {a}}| 
{\mtxx {1} {-1} {0} {1}} \bigg]\\
& +\bigg[{\mtxx {1} {a^{-1}} {0} {1}}| 
{\mtxx {a^{-1}} {0} {0} {a}}| 
{\mtxx {a^{-2}(a-1)^2} {0} {0} {1}} \bigg]
+\bigg[{\mtxx {a} {0} {0} {1}}| 
{\mtxx {1} {a^{-1}(a+1)} {0} {1}}| 
{\mtxx {1} {-1} {0} {1}} \bigg]\\
&+\bigg[{\mtxx {a^{-1}} {0} {0} {a}}| 
{\mtxx {1} {a(1-a^{-2})} {0} {1}}| 
{\mtxx {1} {a^{-1}} {0} {1}} \bigg]
+\bigg[{\mtxx {1} {-a^{-1}} {0} {1}}| 
{\mtxx {a} {0} {0} {1}}| 
{\mtxx {a^{-1}} {0} {0} {a}} \bigg]\\
&+\bigg[{\mtxx {1} {-a^{-1}} {0} {1}}| 
{\mtxx {a^{-1}} {0} {0} {-a^{-1}(a-1)^2}}| 
{\mtxx {a} {0} {0} {a^{-1}}} \bigg]
-\bigg[{\mtxx {1} {a+1} {0} {1}}| 
{\mtxx {a} {0} {0} {1}}| 
{\mtxx {1} {-1} {0} {1}} \bigg]\\
&-\bigg[{\mtxx {1} {a^{-1}} {0} {1}}| 
{\mtxx {a^{-2}(a-1)^2} {0} {0} {1}}| 
{\mtxx {a^{-1}} {0} {0} {a}} \bigg]
+\bigg[{\mtxx {a} {0} {0} {1}}| 
{\mtxx {a^{-1}} {0} {0} {a}}| 
{\mtxx {1} {-1} {0} {1}} \bigg]\\
& +\bigg[{\mtxx {1} {-1} {0} {a}}| 
{\mtxx {1} {-(a-1)^{-1}} {0} {1}}| 
{\mtxx {a^{-1}} {0} {0} {-a^{-1}(a-1)^2}} \bigg]
-\bigg[{\mtxx {a} {0} {0} {1}}| 
{\mtxx {1} {-a^{-2}} {0} {1}}| 
{\mtxx {a^{-1}} {0} {0} {a}} \bigg]\\
&-\bigg[{\mtxx {1} {-(a-1)^{-1}} {0} {1}}| 
{\mtxx {1} {0} {0} {a}}| 
{\mtxx {a^{-1}} {0} {0} {-a^{-1}(a-1)^2}} \bigg]
+\bigg[{\mtxx {1} {-a^{-1}} {0} {1}}| 
{\mtxx {a} {0} {0} {a^{-1}}}| 
{\mtxx {1} {0} {0} {a}} \bigg]\\
&-\bigg[{\mtxx {a} {-1} {0} {1}}| 
{\mtxx {1} {(a-1)^{-1}} {0} {1}}| 
{\mtxx {a^{-1}(a-1)} {0} {0} {a^{-1}(a-1)}} \bigg]
-\bigg[{\mtxx {a} {0} {0} {a^{-1}}}| 
{\mtxx {1} {-a^{-3}} {0} {1}}| 
{\mtxx {1} {0} {0} {a}} \bigg]\\
& +\bigg[{\mtxx {1} {(a-1)^{-1}} {0} {1}}| 
{\mtxx {a} {0} {0} {1}}| 
{\mtxx {a^{-1}(a-1)} {0} {0} {a^{-1}(a-1)}} \bigg]
-\bigg[{\mtxx {a} {0} {0} {1}}| 
{\mtxx {1} {1} {0} {1}}| 
{\mtxx {1} {-1} {0} {1}} \bigg]\\
&+\bigg[{\mtxx {1} {-(a-1)^{-1}} {0} {1}}| 
{\mtxx {a^{-1}} {0} {0} {-a^{-1}(a-1)^2}}| 
{\mtxx {1} {0} {0} {a}} \bigg]
+\bigg[{\mtxx {a^{-1}} {0} {0} {a}}| 
{\mtxx {-1} {0} {0} {a}}| 
{\mtxx {1} {-a(a+1)} {0} {1}} \bigg]\\
&- \bigg[{\mtxx {1} {(a-1)^{-1}} {0} {1}}| 
{\mtxx {a^{-1}(a-1)} {0} {0} {a^{-1}(a-1)}}| 
{\mtxx {a} {0} {0} {1}} \bigg]
+ \bigg[{\mtxx {a^{-1}} {0} {0} {a}}| 
{\mtxx {1} {1} {0} {1}}| 
{\mtxx {1} {-1} {0} {1}} \bigg]\\
& -\bigg[{\mtxx {-a} {0} {0} {a^{-1}}}| 
{\mtxx {1} {-a^{-1}(a+1)} {0} {1}}| 
{\mtxx {1} {a^{-1}(a+1)} {0} {1}} \bigg]
+\bigg[{\mtxx {1} {-1} {0} {1}}| 
{\mtxx {1} {a+1} {0} {1}}| 
{\mtxx {a} {-a} {0} {1}} \bigg]\\
&+\bigg[{\mtxx {1} {a(a+1)} {0} {1}}| 
{\mtxx {-a} {0} {0} {a^{-1}}}| 
{\mtxx {1} {a^{-1}(a+1)} {0} {1}} \bigg]
-\bigg[{\mtxx {1} {0} {0} {a}}| 
{\mtxx {1} {-1} {0} {1}}| 
{\mtxx {1} {-(a-1)^{-1}} {0} {1}} \bigg]\\
& -\bigg[{\mtxx {1} {a(a+1)} {0} {1}}| 
{\mtxx {1} {-a(a+1)} {0} {1}}| 
{\mtxx {-a} {0} {0} {a^{-1}}} \bigg]
+ \bigg[{\mtxx {1} {-1} {0} {1}}| 
{\mtxx {a} {0} {0} {1}}| 
{\mtxx {1} {(a-1)^{-1}} {0} {1}} \bigg]\\
& -\bigg[{\mtxx {1} {-1} {0} {1}}| 
{\mtxx {1} {a(a-1)^{-1}} {0} {1}}| 
{\mtxx {a} {0} {0} {1}} \bigg]
-\bigg[{\mtxx {1} {a(1-a^{-2})} {0} {1}}| 
{\mtxx {1} {a^{-1}(a+1)} {0} 1{}}| 
{\mtxx {1} {-1} {0} {1}} \bigg]\\
& +\bigg[{\mtxx {a} {0} {0} {a^{-1}}}| 
{\mtxx {1} {0} {0} {a}}| 
{\mtxx {1} {-a^{-2}} {0} {1}} \bigg]
+\bigg[{\mtxx {-1} {0} {0} {a}}| 
{\mtxx {1} {-a(a+1)} {0} {1}}| 
{\mtxx {-a} {0} {0} {a^{-1}}} \bigg]\\
\end{align*}
\begin{align*}
& -\bigg[{\mtxx {1} {a+1} {0} {1}}| 
{\mtxx {-1} {0} {0} {a}}| 
{\mtxx {-a} {0} {0} {a^{-1}}} \bigg]
-\bigg[{\mtxx {-1} {0} {0} {a}}| 
{\mtxx {-a} {0} {0} {a^{-1}}}| 
{\mtxx {1} {a^{-1}(a+1)} {0} {1}} \bigg]\\
& -\bigg[{\mtxx {a} {0} {0} {1}}| 
{\mtxx {1} {0} {0} {a}}| 
{\mtxx {1} {1} {0} {1}} \bigg]
+\bigg[{\mtxx {1} {a(a-1)^{-1}} {0} {1}}| 
{\mtxx {1} {-1} {0} {1}}| 
{\mtxx {1} {-(a-1)^{-1}} {0} {1}} \bigg]\\
& +\bigg[{\mtxx {a} {0} {0} {1}}| 
{\mtxx {1} {(a-1)^{-1}} {0} {1}}| 
{\mtxx {1} {-(a-1)^{-1}} {0} {1}} \bigg]
\!-\!\bigg[{\mtxx {1} {a(a-1)^{-1}} {0} {1}}| 
{\mtxx {a} {0} {0} {1}}| 
{\mtxx {1} {-(a-1)^{-1}} {0} {1}} \bigg]\\
& +\bigg[{\mtxx {1} {a(a-1)^{-1}} {0} {1}}| 
{\mtxx {1} {-a(a-1)^{-1}} {0} {1}}| 
{\mtxx {a} {0} {0} {1}} \bigg]
-\bigg[{\mtxx {a} {0} {0} {1}}| 
{\mtxx {1} {-(a-1)^{-1}} {0} {1}}| 
{\mtxx {1} {0} {0} {a}} \bigg]\\
& +\bigg[{\mtxx {1} {-a(a-1)^{-1}} {0} {1}}| 
{\mtxx {a} {0} {0} {1}}| 
{\mtxx {1} {0} {0} {a}} \bigg]
+\bigg[{\mtxx {a} {0} {0} {1}}| 
{\mtxx {1} {0} {0} {a}}| 
{\mtxx {1} {-a(a-1)^{-1}} {0} {1}} \bigg].
\end{align*}
Using the fact that
\begin{align*}
\delta_3([(a^{-1},a)|({-1},a)|(-a,a^{-1})])=& 
\!+\![(-1, a)|(-a, a^{-1})]\!+\![(a^{-1}, a)|(a, 1)]\\
&\!-\![(-a^{-1}\!, a^2)|(-a, a^{-1})]\!-\![(a^{-1}\!, a)|(-1, a)],
\end{align*}
we see that
\begin{gather*}
\begin{array}{rl}
Y_a=
& \!\!\! +{\rm \bf{c}}((-1, a),(-a, a^{-1})) + 2 {\rm \bf{c}}((a^{-1}, a),(a, 1))
+{\rm \bf{c}}((1, a),(a, a^{-1})) \\
&  \!\!\! +{\rm \bf{c}}((a^{-2}(a-1)^2, 1),(a^{-1}, a))
+{\rm \bf{c}}((a, a^{-1}),(a^{-1}, -a^{-1}(a-1)^2)) \\
&  \!\!\! +{\rm \bf{c}}((a^{-1}(a-1), a^{-1}(a-1)),(a, 1))
+{\rm \bf{c}}((1, a),(a^{-1}, -a^{-1}(a-1)^2))\\
= & \!\!\!+{\rm \bf{c}}((a, 1),(1-a, 1)) - {\rm \bf{c}}((a, 1),(1, 1-a))\in H_2(T_2)_\sigma.
\end{array}
\end{gather*}
Also we have
\begin{align*}
Z_a = & +{\rm \bf{c}}\bigg({\mtxx {1} {-1} {0} {1}}, {\mtxx {1} {1} {0} {1}}\bigg)+
{\rm \bf{c}}\bigg({\mtxx {1} {-a^{-2}} {0} {1}}, {\mtxx {1} {a^{-2}} {0} {1}}\bigg)\\
&+{\rm \bf{c}}\bigg({\mtxx {1} {a(a+1)} {0} {1}}, {\mtxx {1} {-a(a+1)} {0} {1}}\bigg)
+{\rm \bf{c}}\bigg({\mtxx {1} {-1} {0} {1}}, {\mtxx {1} {a(a-1)^{-1}} {0} {1}}\bigg)\\
&+{\rm \bf{c}}\bigg({\mtxx {1} {a+1} {0} {1}}, {\mtxx {1} {-1} {0} {1}}\bigg)
+{\rm \bf{c}}\bigg({\mtxx {a} {0} {0} {a}}, {\mtxx {1} {1} {0} {1}}\bigg)\\
&+{\rm \bf{c}}\bigg({\mtxx {1} {-a(a-1)^{-1}} {0} {1}}, {\mtxx {a} {0} {0} {a}}\bigg).
\end{align*}
Now it is easy to see that $d_{3,0}^3([a])=Y_a+Z_a$ corresponds to the element
$(a\wedge (1-a), -a \otimes (1-a), -2{\rm \bf{c}}(a,1-a)) \in H_2(\rr) \oplus 
(\rr \otimes \rr)_\sigma \oplus H_2(R)_\rr$. 
\end{proof}

\begin{cor}\label{vdk--0}
{\rm (i)} Let $R/\mmm_R$ has at least three elements. Then 
\[
H_2(\GL_2(R)) \simeq H_2(\GL_1(R))\oplus \Big((\rr\otimes\rr)_\sigma\oplus H_2(R)_\rr\Big)/K,
%\lan (a \otimes (1-a), 2{\rm \bf{c}}(a,1-a))):a,1-a\in \rr\ran.
\]
where $K=\lan \Big(a \otimes (1-a), 2{\rm \bf{c}}(a,1-a)\Big):a,1-a\in \rr\ran$.

\par {\rm (ii)} If $R/\mmm_R$ has at least four elements, then
\[
H_2(\SL_2(R))_\rr \simeq \Big((\rr\otimes\rr)_\sigma\oplus H_2(R)_\rr\Big)/K.
\]
\end{cor}
\begin{proof}
(i) By an easy analysis of the above spectral sequence and  Lemma~\ref{d3}, one sees that
\begin{align*}
H_2(\GL_2(R)) & \simeq  \Big(H_2(T_2)_\sigma \oplus H_2(R)_\rr\Big)/\im(d_{3,0}^3)\\
& =  \Big(H_2(\rr) \oplus (\rr \otimes \rr)_\sigma\oplus H_2(R)_\rr\Big)/L, 
\end{align*}
where $L=\lan \Big(a\wedge(1-a), -a \otimes (1-a), -2{\bf c}(a,1-a)\Big)| a, 1-a \in \rr \ran$. 
Let $T:=\Big(H_2(\rr) \oplus (\rr \otimes \rr)_\sigma\oplus H_2(R)_\rr\Big)/L$.
From the maps
\begin{align*}
& H_2(\GL_1(R)) \arr T, && x \mt (x, 0,0)+L, \\
& T \arr H_2(\GL_1(R)), && (x, c\otimes d, z)+L \mt x+ c\wedge d,\\
& \Big((\rr\!\otimes\!\rr)_\sigma\!\oplus H_2(R)_\rr\!\Big)/K\! \arr \!T, && 
(a\otimes b,z)\!+\! K \!\mt\! (a\wedge b,\!-a \otimes b, z)\!+\! L, \\
& T \arr \Big((\rr\otimes\rr)_\sigma\oplus H_2(R)_\rr\Big)/K, && 
(x, c\otimes d, z)+L \mt (-c\otimes d,z)+K,
\end{align*}
we obtain the isomorphism 
$T \simeq H_2(\GL_1(R))\oplus \Big((\rr\otimes\rr)_\sigma\oplus H_2(R)_\rr\Big)/K$.
\par (ii) Since $|R/\mmm_R|\geq 4$, $H_1(\SL_2(R))=0$. Then from the corresponding 
Lyndon-Hochschild-Serre spectral sequence of the extension 
\[
1\arr \SL_2(R)\arr \GL_2(R)\arr \rr \arr 1,
\]
it is easy to show that $H_2(\GL_2(R))\simeq  H_2(\GL_1(R))\oplus H_2(\SL_2(R))_\rr$. 
Now the claim follows from (i).
\end{proof}

%%%%%%%%%%%%%%%%%%%%%%%%%%%%%%%%%%%%%%%%%%%%%%%%%%%%%%%%%%%%%%%%%%%%%%%%
\section{Homology of affine groups}\label{aff}
%%%%%%%%%%%%%%%%%%%%%%%%%%%%%%%%%%%%%%%%%%%%%%%%%%%%%%%%%%%%%%%%%%%%%%%%

Lemma \ref{T2-B2} shows that the groups $H_i(\rr, H_m(R^n, \z))$ 
are important in the study of the homology of affine groups.
They already have been studies by Nesterenko and Suslin \cite{suslin1985}, \cite{nes-suslin1990} 
over local rings with infinite residue field and by Hutchinson \cite{hutchinson2013}, 
\cite{hutchinson2014} over a large class of local rings. 

\begin{prp}\label{hutchinson4}
Let $R$ be a local ring. If $R/\mmm_R$ is finite of order $p^{d}$, we suppose 
that $1\leq m <(p-1)d$.
\par  {\rm (i)} If $m=1$ or $m=2$, then for any $r \geq 0$, $H_r(\rr, H_m(R^n, \z))=0$.
\par {\rm (ii)} For any prime field $k$ and any $r \geq 0$, $H_r(\rr, H_m(R^n, k))=0$.
\par {\rm (iii)} If $R$ is a domain or an algebra over a field, then for any $r \geq 0$,
$H_r(\rr, H_m(R^n, \z))=0$.
\end{prp}
\begin{proof}
Part (i) follows directly from Lemma \cite[Lemma~3.17]{hutchinson2014}.
The proof of (ii) is similar to the proof of \cite[Proposition~1.10]{nes-suslin1990} 
or \cite[Proposition~1.7]{suslin1985}. For (iii) see \cite[Lemma~3.18]{hutchinson2014}.
In fact, Hutchinson proved (iii) for local domains, but his arguments also works for local 
algebras over fields.
\end{proof}

Let $G_m(R)$ be a subgroup of $\GL_m(R)$ and $G_n(R)$ a subgroup of $\GL_n(R)$ and  
assume that either $\rr I_m \se G_m(R)$ or $\rr I_n \se G_n(R)$. 
Let $M(R)$ be a free submodule of $M_{m,n}(R)$ such that $G_m(R)M(R)=M(R)=M(R)G_n(R)$. Then 
$A_{m,n}(R):={\mtx {G_m(R)} {M(R)} {0} {G_n(R)}}$ is a subgroup of the affine group
$\Aff_{m,n}(R):={\mtx {\GL_m(R)} {M_{m,n}(R)} {0} {\GL_n(R)}}$.

\begin{prp}\label{affine}
Let $R$ be a local ring. If $R/\mmm_R$ is finite, we assume that it is of order $p^{d}$. 
Let the natural homomorphism
\[
\phi_q: H_q(G_m(R) \times G_n(R)) \arr H_q(A_{m,n}(R))
\]
be induced by the inclusion $G_m(R) \times G_n(R) \arr A_{m,n}(R)$.
%, where $G_m(R)$, $G_n(R)$ and $A_{m,n}(R)$ are as above.
\par {\rm (i)} For $0 \leq q \leq 2$, $\phi_q$ is an isomorphism if $q < p-1)d$. 
So $\phi_0$ always is an isomorphism, $\phi_1$  is an isomorphism if 
$|R/\mmm_R|\neq 2$ and $\phi_2$ is an isomorphism if $|R/\mmm_R|\neq 2, 3, 4$.
\par {\rm (ii)} If $R$ is a domain or an algebra over a field, 
then $\phi_q$ is an isomorphism for $0 \leq q < (p-1)d$. 
\par {\rm (iii)} The map  $\phi_q$ is an isomorphism for $0 \leq q <(p-1)d-2$. 
In particular, if $R/\mmm_R$ is infinite, then $\phi_q$ is an isomorphism for any $q$.
\end{prp}
\begin{proof}
This can be done as the proof of \cite[Theorem~1.9]{suslin1985}.
For (iii) we also need the next lemma.
\end{proof}

\begin{lem}\label{isom} 
Let $f: H\arr G$ be a homomorphism of groups and let $n$ be a positive integer.
If $H_i(f): H_i(H, k) \arr H_i(G,k)$ is an isomorphism for any $0 \leq i \leq n$,
and any prime field $k$, then $H_i(f): H_i(H) \arr H_i(G)$ is an isomorphism 
for any $0 \leq i \leq n-2$.
\end{lem}
\begin{proof}
Let $p$ be a prime. Then from the long exact sequence of the homology of $H$ and $G$
applied to $0 \arr \z/p^{d-1} \arr \z/p^{d} \arr \z/p \arr 0$ and by induction on $d$, we 
see that $H_i(H, \z/p^{d}) \arr H_i(G,\z/p^{d})$ is an isomorphism for 
any $0 \leq i \leq n-1$ and any $d \geq 1$. This together with the fact that 
$\q/\z\simeq \bigoplus_{p \ prime} \underset{\larr}{\lim}\ \z/p^{d}$, imply that
$H_i(H, \q/\z) \arr H_i(G, \q/\z)$ is an isomorphism for any $0 \leq i \leq n-1$.
Now by applying a similar method to the exact sequence 
$0 \arr \z \arr \q \arr \q/\z \arr 0$, we see that $H_i(H,\z) \arr H_i(G, \z)$ is an
isomorphism for any $0 \leq i \leq n-2$.
\end{proof}

\begin{exa}\label{finite-fields}
Let $B_2(\F_q)$ and $T_2(\F_q)$
denote $B_2$ and $T_2$ over the finite field $R=\F_q$. If $q\neq 2, 3, 4, 8$, 
then by Proposition \ref{affine}, $H_i(B_2(\F_q))\simeq H_i(T_2(\F_q))$ for 
$0\leq i\leq 3$.
In this example we will discuss these groups when $q=2, 3,4,8$. Note that 
%by Lemma~\ref{h2-h3},
\[
\begin{array}{c}
H_2(\F_q)\simeq \bigwedge_\z^2 \F_q \ \ \ \text{and}\ \ \ 
H_3(\F_q) \simeq \bigwedge_\z^3 \F_q \oplus (\F_q \otimes \F_q)^{-\sigma},
\end{array}
\]
where $-\sigma(a\otimes b)=b\otimes a$ (see \cite[Lemma~5.5]{suslin1991}, 
\cite[p.~38]{hutchinson2013}). 
%The spectral sequence $\EE_{r,s}^2$ is the one that appears in the proof of Lemma \ref{B2-T2}.
\medskip

(i) $R=\F_2$: In this case $B_2(\F_2)\simeq \z/2$ and $T_2(\F_2)=\{1\}$. Thus
\[
H_1(B_2(\F_2))\simeq H_1(T_2(\F_2))\oplus \z/2, \ \ \ \ 
H_2(B_2(\F_2))\simeq H_2(T_2(\F_2))
\]
and
\[ 
H_3(B_2(\F_2))\simeq  H_3(T_2(\F_2)) \oplus \z/2.
\]

(ii) $R=\F_3$: Since $H_2(\F_3)=0$, for any $r\geq 0$ we have $\EE_{r,2}^2=0$,
where this spectral sequence was discussed in the proof of Lemma \ref{T2-B2}. Moreover,
$H_3(\F_3)\simeq (\F_3 \otimes \F_3)^{-\sigma}=\{0,1\otimes 1, 1\otimes 2\}$
and by a direct computation we see that the action of $\F_3^\times$ on $H_3(\F_3)$ 
is trivial (note that $\F_3^\times$ acts diagonally on $(\F_3 \otimes \F_3)^{-\sigma})$.
Thus $H_3(\F_3)_{\F_3^\times}\simeq \z/3$. 
Now from the spectral sequence $\EE_{r,s}^2$ we get
\[
H_1(B_2(\F_3))\simeq H_1(T_2(\F_3)), \ \ \ \
H_2(B_2(\F_3))\simeq H_2(T_2(\F_3))
\]
and
\[
H_3(B_2(\F_3))\simeq H_3(T_2(\F_3)) \oplus \z/3.
\]

(iii) $R=\F_4$:  Clearly $H_2(\F_4)\simeq \bigwedge_\z^2\F_4 \simeq \z/2$
and thus  $H_2(\F_4)_{\F_4^\times}\simeq \z/2$. These show that
the action of $T_2(\F_4)$ on $H_2(\F_4)$ is trivial. Thus using the Universal Coefficient
Theorem one can show that $\EE_{1,2}^2$ and $\EE_{2,2}^2$ are trivial.
By applying the K\"unneth formula to $(\F_4, +)\simeq \z/2 \oplus \z/2$, one sees that
$H_3(\F_4)$ has $8$ elements. Thus if  $\F_4=\{0,1, \alpha, \alpha+1\mid \alpha^2= \alpha+1\}$, 
then
\begin{align*}
H_3(\F_4)\simeq (\F_4 \otimes \F_4)^{-\sigma} =\lan & 1\otimes 1, \alpha \otimes \alpha,
1 \otimes \alpha+\alpha \otimes 1\ran\\
=\{ & 0, 
1\otimes 1, 
\alpha \otimes \alpha,
1 \otimes \alpha+\alpha \otimes 1,
1 \otimes 1 + \alpha \otimes \alpha,\\ 
& 1\otimes 1 + 1 \otimes \alpha + \alpha \otimes 1, 
\alpha \otimes \alpha + 1 \otimes \alpha + \alpha \otimes 1,\\
& 1\otimes 1 +\alpha \otimes \alpha + 1 \otimes \alpha + \alpha \otimes 1\}.
\end{align*}
Now by a direct calculation we have 
$H_3(\F_4)_{\F_4^\times}\simeq 
\Big((\F_4 \otimes \F_4)^{-\sigma}\Big)_{\F_4^\times}\simeq \z/2$.
Finally from the spectral sequence $\EE_{r,s}^2$ we get the isomorphisms
\[
H_1(B_2(\F_4))\simeq H_1(T_2(\F_4)), \ \ \ \
H_2(B_2(\F_4))\simeq H_2(T_2(\F_4)) \oplus \z/2
\]
and
\[
H_3(B_2(\F_4))\simeq H_3(T_2(\F_4)) \oplus \z/2.
\]

(iv) $R=\F_8$: Here we need to compute $H_3(\F_8)_{\F_8^\times}$. We have
\[
\begin{array}{c}
H_3(\F_8) \simeq \bigwedge_\z^3 \F_8 \oplus (\F_8 \otimes \F_8)^{-\sigma}
\simeq \z/2 \oplus (\F_8 \otimes \F_8)^{-\sigma}.
\end{array}
\]
Again using the K\"unneth formula one sees that $H_3(\F_8)$ has $2^7$ elements.
Thus $(\F_8 \otimes \F_8)^{-\sigma}$ has $2^6$ elements. If we assume
\[
\F_8=\{0, 1, \alpha, \alpha+1, \alpha^2, \alpha^2+1,  \alpha^2+ \alpha,  \alpha^2+ \alpha+1 
\mid \alpha^3=\alpha+1\},
\]
then 
\begin{align*}
(\F_8 \otimes \F_8)^{-\sigma}= \lan & 1\otimes 1, 
\alpha \otimes \alpha, \alpha^2 \otimes \alpha^2,
1 \otimes \alpha+\alpha \otimes 1, 1 \otimes \alpha^2+\alpha^2 \otimes 1,\\
&\alpha^2 \otimes \alpha+\alpha \otimes \alpha^2\ran.
\end{align*}
Now by a direct computation one sees that 
$\Big((\F_8 \otimes \F_8)^{-\sigma}\Big)_{\F_8^\times}=0$. Therefore 
$H_3(\F_8)_{\F_8^\times}\simeq \z/2$ and hence
\[
H_1(B_2(\F_8))\simeq H_1(T_2(\F_8)), \ \ \ \ H_2(B_2(\F_8))\simeq H_2(T_2(\F_8))
\]
and 
\[
H_3(B_2(\F_8))\simeq H_3(T_2(\F_8)) \oplus \z/2.
\] 
%These computations show that the assumption made on the size of $R/\mmm_R$
%in Proposition \ref{B2-} can not be improved.
\end{exa}

\begin{exa}\label{sln-finite-fields}
Let $R=\F_q$ be the finite field with $q$ elements. If $q\geq 5$, then by 
Corollary \ref{vdk--0} and Proposition 
\ref{hutchinson4} we have $H_2(\SL_2(\F_q))_{\F_q^\times}\simeq K_2^{M}(\F_q)=0$.
If $q=4$, then $H_2(\F_4)_{\F_4^\times}\simeq \z/2$ and $K_2^{M}(\F_4)=0$. 
Thus by Corollary \ref{vdk--0}, we have
\[
H_2(\SL_2(\F_4))_{\F_4^\times} \simeq H_2(\F_4)_{\F_4^\times}\simeq \z/2.
\]
If $q=3$, then $H_2(\F_3)=H_2(\F_3^\times)=0$ and
$K_2^{M}(\F_3)=0$. Now by Corollary~\ref{vdk--0}, $H_2(\GL_2(\F_3))=0$. 
If $q=2$, then $\GL_2(\F_2)=\SL_2(\F_2)\simeq \Sigma_3$.
By looking at the associated 
Lyndon-Hochschild-Serre spectral sequence ${E''}_{p,q}^2$ of the extension 
\[
1 \arr A_3 \arr \Sigma_3 \arr \Sigma_3/A_3 \arr 1,
\]
for any pair $(p,q)=(p,2s)$ or 
$(p,q)=(2r,0)$, we have ${E''}_{p,q}^2=0$. Since the action of $\Sigma_3/A_3$ on 
$A_3$ is non-trivial, ${E''}_{0,1}^2=H_0(\Sigma_3/A_3, A_3)=0$. Thus by Lemma \ref{dwyer} below,
%\cite[Theorem~1]{dwyer1975}, 
${E''}_{p,1}^2=H_p(\Sigma_3/A_3, A_3)=0$ for all 
$p\geq 0$. An easy analysis of the above spectral sequence implies that 
$H_1(\SL_2(\F_2))\simeq \z/2$, $H_2(\SL_2(\F_2))=0$. Therefore, 
%using Proposition \ref{vdk}, we can prove the following well-known result
\[
H_2(\GL_2(\F_q))=\begin{cases}
0& \text{if $q \neq 4$}\\
\z/2 & \text{if $q = 4$.}
\end{cases}
\]
\end{exa}

\begin{lem}\label{dwyer}
Let $G$ be an abelian group and $M$ a finitely generated $G$-module such that $H_0(G,M)=0$.
Then for any $n\geq 0$, $H_n(G,M)=0$. 
\end{lem}
\begin{proof}
See \cite[p.~11--12]{dwyer1975}.
\end{proof}

For a subgroup $H$ of a group $G$ and a $G$-module $M$, the natural map $H_n(H,M) \arr H_n(G, M)$
is called the corestriction map and is denoted by 
\[
\corr_G^H: H_n(H,M) \arr H_n(G, M).
\]
When the index of $H$ in $G$ is finite, i.e.  $[G:H]\leq \infty$, for any $n\geq 0$
there is a restriction map, called  transfer map,
\[
\res_H^G: H_n(G,M) \arr H_n(H, M),
\]
such that
\[
\corr_G^H\circ \res_H^G=[G:H]\id_{H_n(G,M)},
\]
\cite[Proposition~9.5, Chap.~III]{brown1994}. 

In case $G$ is finite, by putting $H=\{1\}$, one sees that $H_n(G,M)$ is 
annihilated by $|G|$ for all $n>0$ \cite[Corollary~10.2, Chap.~III]{brown1994}. 
It is well known that when $M$ is a finitely generated $G$-module, then 
$H_n(G,M)$ is finite for all $n>0$.
% \cite[Theorem~10.3, Chap.~III]{weibel1994}. 

Let $G$ be a finite group. For a $g \in G$, let $H_n(H) \arr H_n(gHg^{-1})$, 
$z \mt g.z$, be induced by the natural map $H \arr gHg^{-1}$. We say $z\in H_n(H)$
is $g$-invariant if
\[
\res_{H\cap gHg^{-1}}^H(z)=\res_{H\cap gHg^{-1}}^{gHg^{-1}}(g.z).
\]
Let 
\[
\inv_G(H_n(H)):=\Big\{z\in H_n(H)\mid \text{$z$ is $g$-invarinat for all $g \in G$}\Big\}.
\]
If $H$ is a $p$-Sylow subgroup of $G$, then one can show that 
\[
H_n(G)_{(p)}\simeq \inv_G(H_n(H)),
\]
where $H_n(G)_{(p)}$ is the $p$-primary component of $H_n(G)$ 
\cite[Theorem~10.3, Chap.~III]{brown1994}. 
Moreover, if $H$ is normal in $G$, then 
\[
H_n(G)_{(p)}\simeq \inv_G(H_n(H))\simeq H_n(H)_{G/H}.
\]
For $g \in G$ and $z\in H_n(H)$ the condition 
$\res_{H\cap gHg^{-1}}^H(z)=\res_{H\cap gHg^{-1}}^{gHg^{-1}}(g.z)$ is 
trivially satisfied if $H\cap gHg^{-1}=1$. Thus to determine $\inv_G(H_n(H))$
for a $p$-Sylow subgroup $H$, it is enough to consider only the set 
$\Conj(G,H)$ of those elements $g$ for which $H\cap gHg^{-1}\neq 1$:
\[
\Conj(G,H):=\{g \in G \mid H\cap gHg^{-1}\neq 1\}.
\]

\begin{exa}\label{p-sylow}
Let $F=\F_{p^m}$, where $p$ is a prime.
Let $N_2(F)={\mtx 1 F 0 1}\simeq F$. Then $N_2(F)$ is a $p$-Sylow subgroup of 
$\GL_2(F)$. It is easy to see that $A \in \Conj(\GL_2(F), N_2(F))$ if and only 
if $A \in B_2(F)$. Thus
\[
\Conj(\GL_2(F), N_2(F))=B_2(F)=\Conj(B_2(F), N_2(F)),
\]
and hence
\begin{align*}
H_3(\GL_2(F))_{(p)}& \simeq \inv_{\GL_2(F)}(H_3(N_2(F)))\\
&=\inv_{B_2(F)}(H_3(N_2(F)))\simeq 
H_3(B_2(F))_{(p)}.
\end{align*}
Now by Example \ref{finite-fields}, we have
\[
H_3(\GL_2(\F_{p^m}))_{(p)}=\begin{cases}
\z/p & \text{if $p^m=2,3,4,8$}\\
0    & \text{otherwise.}
\end{cases}
\]
\end{exa}

%%%%%%%%%%%%%%%%%%%%%%%%%%%%%%%%%%%%%%%%%%%%%%%%%%%%%%%%%%%%%%%%%%%%%%%%%%%%%%%%
\section{Stability for the second homology group and the second {\it K}-group}\label{H2}
%%%%%%%%%%%%%%%%%%%%%%%%%%%%%%%%%%%%%%%%%%%%%%%%%%%%%%%%%%%%%%%%%%%%%%%%%%%%%%%%

The homology stability results have many important application in algebraic $K$-theory.
In this article the homology stability for the second and the third homology of the general 
linear group play very important rolls in proving our main results. 

Let  $C_l'(R^n)$ be the free abelian group with a basis consisting of 
$(l+1)$-tuples $(w_0, \dots, w_l)$, where every
$\min\{l+1, n\}$ of  $w_i \in R^n$ are basis of a free direct summand 
of $R^n$ and consider it as $\GL_n(R)$-module in a natural way. Let us 
define the differential operators $\partial_l' : C_l'(R^n) \arr C_{l-1}'(R^n)$
similar to those in the complex $C_\bullet(R^2)$. So we have the complex of $\GL_n(R)$-modules:
\begin{align*}
C_\bullet'(R^n):& \ \
\cdots  \arr C_{n-1}'(R^n) \arr C_{n-2}'(R^n)  \arr \cdots \arr
C_0'(R^n) \arr C_{-1}'(R^n) \arr 0,
\end{align*}
where $C_{-1}'(R^n)=\z$.

\begin{lem}\label{suslin-vdkallen}
%Let $R$ be a local ring. 
The complex  $C_\bullet'(R^n)$ is exact for $-1 \leq  i \leq n-2$.
\end{lem}
\begin{proof}
This follows from \cite[\S 2, Theorem]{vdkallen1980}.
\end{proof}

Now as before, 
if the residue field of $R$ is finite, we assume that $|R/\mmm_R|=p^d$. Let $s < (p-1)d-2$. 
If $R$ is a domain or an algebra over a field, we may only assume that $s <(p-1)d$.

Let $L_l':=C_{l-1}'(R^n)$ and let $\mathcal{C}_\bullet(\GL_n(R)) \arr  \z$ be the 
standard resolution of $\GL_n(R)$. From the double complex 
$\mathcal{C}_\bullet(\GL_n(R)) \otimes _{\GL_n(R)} L_\bullet'$
we obtain the first quadrant spectral sequence
\[
{E}_{r, s}^1(n)= H_s(\GL_n(R), L_{r}') \Rightarrow H_{r+s}(\GL_n(R), L_\bullet'),
\]
(see Section \ref{SP}).
It follows from Lemmas \ref{suslin-vdkallen} and \ref{homology-complex} that 
for $0 \leq m \leq n-1$,
\[
H_m(\GL_n(R), L_\bullet')=0.
\]
Let $\sigma_r:=(e_1, \dots, e_r)\in L_r'$, $1 \leq r \leq n$. Then by the Shapiro lemma we have
\[
{E}_{r, s}^1(n):=H_s(\GL_n(R), L_r')\simeq H_s(\stabe_{\GL_n(R)}(\sigma_r)),
\]
where $\stabe_{\GL_n(R)}(\sigma_r)={\mtx {I_r}  {M_{r,n-r}(R)} {0} {\GL_{n-r}(R)}}$.
%Now let $s < (p-1)d-2$. If $R$ is a domain or is an algebra over a field, we may only 
%assume that $s < (p-1)d$. 
Then by Proposition~\ref{affine} we have
\[
{E}_{r, s}^1(n)=H_s(\stabe_{\GL_n(R)}(\sigma_r))\simeq H_s(\GL_{n-r}(R)).
\]
Moreover, it is not difficult to see that for $1\leq r \leq n$, the differential
%If $s < (r-1)d$, then by Lemma \ref{affine}, ${d}_{r, s}^1(n)$ is of the form
\[
{d}_{r, s}^1(n): H_s(\GL_{n-r}(R)) \arr H_s(\GL_{n-r+1}(R))
\] 
%Now it is not difficult to see that, for $1 \le r \le n$ and $0 \leq s < (p-1)d$, we have
is defined as
\[
{d}_{r, s}^1(n)
=\sum_{i=1}^r(-1)^{i+1}H_s(\inc)= \left\{
\begin{array}{ll}
{\rm H_s(inc)} & \textrm{if $r$ is odd}\\
0 & \textrm{if $r$ is even,}
\end{array}
\right.
\]
where $\inc: \GL_{n-r}(R) \arr \GL_{n-r+1}(R)$ is the inclusion map 
\cite[\S~4]{vdkallen1980}, \cite[Lemma~2.4]{nes-suslin1990}. 

\begin{thm}\label{stability1}
Let $R$ be a local ring. If  $R/\mmm_R$ is finite we assume that it has $p^d$ elements. 
\par {\rm (i)} If $|R/\mmm_R|\neq 2,3,4$, then 
\[
H_2(\GL_2(R)) -\!\!\!\two H_2(\GL_3(R)) \overset{\simeq}{\larr} H_2(\GL_4(R)) 
\overset{\simeq}{\larr} \cdots.
\]
\par {\rm (ii)} If $n <(p-1)d$ and $R$ is a domain or an algebra over a field , then 
\[
H_n(\GL_n(R)) -\!\!\!\two H_n(\GL_{n+1}(R)) \overset{\simeq}{\larr} H_n(\GL_{n+2}(R)) 
\overset{\simeq}{\larr}\cdots.
\]
\par {\rm (iii)} If $n <(p-1)d-2$, then 
\[
H_n(\GL_n(R)) -\!\!\!\two H_n(\GL_{n+1}(R)) \overset{\simeq}{\larr} H_n(\GL_{n+2}(R)) 
\overset{\simeq}{\larr}\cdots.
\]
\end{thm}
\begin{proof}
Here one can argue as in  \cite[pp. 127--128]{nes-suslin1990}.
\end{proof}

\begin{rem}
An unpublished work of Quillen shows that if $F$ is any field with more than two 
elements, then for any $n$ we have the homology stability
\[
H_n(\GL_n(F)) -\!\!\!\two H_n(\GL_{n+1}(F)) \overset{\simeq}{\larr} H_n(\GL_{n+2}(F)) 
\overset{\simeq}{\larr}\cdots.
\]
It is nice and important to know that whether such a result is true when
$F$ is replaced by a local ring $R$ such that $R/\mmm_R$ has more than two elements.
\end{rem}

Now we will show that when $|R/\mmm_R|\neq 2,3,4$, then 
the surjective map $H_2(\GL_2(R)) \two H_2(\GL_3(R))$ is injective.

Let $\widehat{C}_l(R^3)$, $l \geq 0$, be the free abelian group with
a basis consisting of $(l+1)$-tuple $(\lan v_0\ran, \dots, \lan v_l\ran)$,
where every $\min\{l+1, 2\}$ of $v_i \in R^n$ are basis of a free direct 
summand of $R^n$. We define the differential
$\widehat{\partial}_l : \widehat{C}_l(R^3) \arr \widehat{C}_{l-1}(R^3)$, $l \geq 0$,
similar to differentials of $C_\bullet(R^2)$ and we construct the complex 
\[
\widehat{C}_\bullet(R^3): \ \ \ \ \
\cdots \arr  
\widehat{C}_2(R^3) \overset{{\hat{\partial}}_2}{\larr}
\widehat{C}_1(R^3) \overset{{\hat{\partial}}_1}{\larr}
\widehat{C}_0(R^3) \overset{{\hat{\partial}}_0}{\larr}
\widehat{C}_{-1}(R^3)=\z \arr 0
\]
in usual way. 

\begin{lem}\label{R3-exact}
The complex $\widehat{C}_\bullet(R^3)$ is exact for $-1 \leq i < |\PP^{2}(R/\mmm_R)|-1$.
In particular, for any local ring $R$, $\widehat{C}_\bullet(R^3)$ is exact for $-1\leq i < 6$.
\end{lem}
\begin{proof}
%The proof of the firs is very similar to the proof of Lemma \ref{hutchinson1}
%given in \cite{hutchinson2014}. We leave the details for the interested readers.
%The second claim follows from the fact that $|\PP^{2}(R/\mmm_R)|\geq |\PP^{2}(\F_2)|=7$.
The proof of this lemma is similar to the proof of Lemma \ref{hutchinson1}
given in \cite{hutchinson2014}.
%The proof given here is similar to the proof of \cite[Lemma~3.21]{hutchinson2014}.
Let $k:=R/\mmm_R$ and consider $\PP^{2}(k)=\{\lan \overline{v}\ran\mid \overline{v}\in k^2-\{0\}\}$.
For any finite subset $S$ of $\PP^{2}(k)$, let $D_l(S)$ be the subgroup of 
$\widehat{C}_l(R^3)$ generated by the generators $(\lan v_0\ran ,\dots, \lan v_l\ran)$, such that
$S\se \{\lan \overline{v}_0 \ran, \dots, \lan \overline{v}_l \ran\}$.
For each $x\in \PP^{2}(k)$, choose $v_x \in R^3$ such that $\lan v_x \ran$ is a
direct summand of $R^3$ and $\lan \overline{v}_x \ran=x$.
For $l \geq 0$, define $s_x:\widehat{C}_l(R^3) \arr \widehat{C}_{l+1}(R^3)$ as 
\[
s_x(\lan v_0 \ran,\dots, \lan v_l\ran)=\begin{cases}
(\lan v_x \ran, \lan v_0 \ran,\dots, \lan v_l\ran)& \text{if 
$x \notin\{\lan\overline{v}_0 \ran,\dots,\lan\overline{v}_l \ran \}$}\\
0 & \text{otherwise.}
\end{cases}
\]
Thus if $(\lan v_0 \ran,\dots, \lan v_l \ran)$ is a generator of $\widehat{C}_l(R^3)$ and 
$x \notin\{\lan \overline{v}_0 \ran, \dots, \lan\overline{v}_l \ran \}$, then
\[
\widehat{\partial}_{l+1}\circ s_x(\lan v_0 \ran,\dots, \lan v_l \ran)
=(\lan v_0 \ran,\dots, \lan v_l\ran)-s_x\circ\widehat{\partial}_l(\lan v_0 \ran,\dots, \lan v_l\ran).
\]
Suppose $x_i \in \PP^{2}(k)$, $0\leq i \leq l$, are disjoints and choose
$v_{x_i} \in R^3$ such that $\lan \overline{v}_{x_i} \ran=x_i$.

Let $z=z'+z'' \in \ker(\widehat{\partial}_{l-1})$, where $z'$ is 
generated by terms that belong to $D_{l-1}(\{x_0\})$ and  $z''$ is 
generated by terms which do not belong to $D_{l-1}(\{x_0\})$. Then 
\begin{align*}
\widehat{\partial}_l\circ s_{x_0}(z)&=\widehat{\partial}_l\circ s_{x_0}(z'')\\
&= z''+s_{x_0}\circ\widehat{\partial}_{l-1}(z'')\\
%=z +s_{x_0}\circ\widehat{\partial}_{l-1}(z'')-z'\\
&=z +s_{x_0}\circ\widehat{\partial}_{l-1}(z)-z'-s_{x_0}\circ\widehat{\partial}_{l-1}(z')\\
&=z -z'-s_{x_0}\circ\widehat{\partial}_{l-1}(z').
\end{align*}
Clearly $-z'-s_{x_0}\circ\widehat{\partial}_{l-1}(z')\in D_{l-1}(\{x_0\})$. Thus
\[
(\widehat{\partial}_l\circ s_{x_0}-\id_{\widehat{C}_{l-1}(R^3)})(z)=z_0,
\]
where $z_0\in D_{l-1}(\{x_0\})$ and clearly $z_0\in \ker(\widehat{\partial}_{l-1})$. 
In a similar way
we have $(\widehat{\partial}_l\circ s_{x_1}-\id_{\widehat{C}_{l-1}(R^3)})(z_0)=z_1$, 
where $z_1\in D_{l-1}(\{x_0,x_1\})$
and $z_1\in \ker(\widehat{\partial}_{l-1})$. 
Repeating this process we get
\[
z_l=(\widehat{\partial}_l\circ s_{x_l}-\id_{\widehat{C}_{l-1}(R^3)})\circ\dots\circ
(\widehat{\partial}_l\circ s_{x_0}-\id_{\widehat{C}_{l-1}(R^3)})(z),
\]
where $z_l\in D_{l-1}(\{x_0,\dots, x_{l}\})=0$ and thus $z_l=0$. From the above formula
we have $\widehat{\partial}_{l-1}(y)+(-1)^{l}z=0$ for some $y$ and therefore 
$\widehat{\partial}_{l}((-1)^{l-1}y)=z$.
The last claim follows from the fact that $|\PP^{2}(k)|\geq |\PP^{2}(\F_2)|=7$.
\end{proof}

Set $\widehat{L}_l:=\widehat{C}_{l-1}(R^3)$ for $l \geq 0$ and consider the complex
\[
\widehat{L}_\bullet: \ \ \ \cdots \arr \widehat{L}_2 \arr \widehat{L}_1 \arr 
\widehat{L}_0 \arr 0.
\]
Then we have the first quadrant spectral sequence
\[
\widehat{E}_{r, s}^1(n)= H_s(\GL_3(R), \widehat{L}_{r}) 
\Rightarrow H_{r+s}(\GL_3(R), \widehat{L}_\bullet).
\]
By Lemmas \ref{R3-exact} and \ref{homology-complex}, for $0 \leq m \leq 6$, we have
$H_m(\GL_3(R), \widehat{L}_\bullet)=0$.
%If the residue field of $R$ is finite, we assume that $|R/\mmm_R|=p^d$. Let $s < (p-1)d-2$. 
%If $R$ is a domain or an algebra over a field, we may only assume that $s <(p-1)d$. 
Moreover, by Proposition \ref{affine},
\[
\widehat{E}_{r, s}^1= \left\{
\begin{array}{ll}
H_s(\rr^{r} \times \GL_{3-r}(R)) & \textrm{if $0 \le r \le 2$}\\
H_s(\GL_3(R), \widehat{C}_{r-1}(R^3)) & \textrm{if $r\ge 3.$}
\end{array}
\right.
\]

\begin{lem}\label{mir-elb2}
{\rm (i)} Let $|R/\mmm_R|\neq 2,3, 4$. Then  the complex
\[
H_2(\rr^2\times
\GL_1(R)) \stackrel{\alpha_\ast-\inc_\ast}{-\!\!\!-\!\!\!-\!\!\!\larr} 
H_2(\rr \times \GL_2(R)) \stackrel{\inc_\ast}{\larr} H_2(\GL_3(R)) \arr 0
\]
is exact, where $\alpha(a,b,c)=(b,a,c)$.
\par {\rm (ii)} Let $|R/\mmm_R|\neq 2, 3, 4, 5, 8, 9, 16, 32$. 
If $R$ is a domain or an algebra over a field, we only may assume that
$|R/\mmm_R|\neq 2, 3, 4, 8$. Then we have the exact sequence
\[
H_3(\rr^2\times \GL_1(R)) \stackrel{\alpha_\ast-\inc_\ast}{-\!\!\!-\!\!\!-\!\!\!\larr} 
H_3(\rr \times \GL_2(R)) \stackrel{\inc_\ast}{\larr} H_3(\GL_3(R)) \arr 0.
\]
%is exact, where $f_2={\alpha_{1,2}}_\ast-{\inc}_\ast$ and $f_1=\inc_\ast$.
\end{lem}
\begin{proof} 
This can be proved as \cite[Corollary 3.5]{mirzaii-2008}. In fact the claims follow
from analysis of the above spectral sequence as done in \cite[\S~3]{mirzaii-2008}.
%Here we need to use Proposition \ref{affine}.
\end{proof}

\begin{prp}\label{mir-stability}
Let $R$ be a local ring.
%Let $|R/\mmm_R|\neq 2,3,4$. Then
\par {\rm (i)} If $|R/\mmm_R|\neq 2,3,4$, then we have the homology stability
\[
H_2(\GL_2(R)) \overset{\simeq}{\larr} H_2(\GL_3(R)) \overset{\simeq}{\larr} H_2(\GL_4(R)) 
\overset{\simeq}{\larr} \cdots.
\]
\par {\rm (ii)} If $|R/\mmm_R|\neq 2,3,4$, then $K_2^M(R)\simeq K_2(R)$ and the natural
 map $K_2^{M}(R) \arr H_2(\GL_2(R))$ 
is given by
$\{a,b\}\mt {\bf c}\bigg({\mtx a 0 0 1}, {\mtx b 0 0 {b^{-1}}}\bigg)$.
\par {\rm (iii)} {\rm (Van der Kallen)} If $|R/\mmm_R| > 5$, then 
\[
K_2(R)\simeq K_2^M(R)\simeq (\rr \otimes \rr)/\lan a \otimes(1-a): a, 1-a \in \rr\ran.
\]
\end{prp}
\begin{proof}
(i) By Proposition~\ref{stability1}, we must prove that the map $H_2(\GL_2(R)) \arr H_2(\GL_3(R))$
is injective. By Lemma \ref{mir-elb2}, we have the exact sequence
\[
H_2(\rr^2\times \GL_1(R)) \stackrel{\alpha_\ast-\inc_\ast}{-\!\!\!-\!\!\!-\!\!\!\larr} 
H_2(\rr \times \GL_2(R)) \stackrel{\inc_\ast}{\larr} H_2(\GL_3(R)) \arr 0
\]
%where $\alpha: \rr^2\times \GL_1(R)) \arr \rr \times \GL_2(R)$ is given by
%$\diag(a,b,c)\mt \diag(b,a,c)$. 
and by the Universal Coefficient Theorem,
\begin{align*}
H_2(\rr \times \GL_2(R))&\simeq H_2(\rr)\oplus H_2(\GL_2(R))\oplus \rr\otimes H_1(\GL_2(R)),\\
H_2(\rr^2\times \GL_1(R))& \simeq H_2(\rr)\oplus H_2(\rr)\oplus H_2(\GL_1(R)) 
\oplus \rr \otimes \rr \\
& \oplus \rr \otimes H_1(\GL_1(R)) \oplus \rr \otimes H_1(\GL_1(R))
\end{align*}
If $x\in \ker\Big(H_2(\GL_2(R)) \arr H_2(\GL_3(R))\Big)$, then 
\[
(0,x,0) \in \ker\Big(\inc_\ast:H_2(\rr \times \GL_2(R)) \arr H_2(\GL_3(R))\Big).
\]
Thus there is
\[
z=({\bf c}(a_1,b_1), {\bf c}(a_2,b_2), x', a\otimes b, c\otimes d, e\otimes f)
\in H_2(\rr^2\times \GL_1(R))
\]
such that $({\alpha}_\ast-\inc_\ast)(z)=(0,x,0)$ and so
\begin{align*}
0=&-{\bf c}(a_1,b_1)+{\bf c}(a_2,b_2),\\
x=&+{\bf c}(\diag(a_1,1),\diag(b_1,1))-{\bf c}(\diag(a_2,1),\diag(b_2,1))\\
&+{\bf c}(\diag(c,1),\diag(1, d))-{\bf c}(\diag(e,1),\diag(1,f)),\\
0=&-b\otimes a-a\otimes b-c\otimes d+e\otimes f.
\end{align*}
All these imply that 
\[
x=-{\bf c}(\diag(b,1),\diag(1, a))-{\bf c}(\diag(a,1),\diag(1,b))=0.
\] 
\par (ii) 
Since $|R/\mmm_R|\geq 5$, Proposition~\ref{hutchinson4} implies that 
$H_2(R)_\rr=0$. Thus by Corollary~\ref{vdk--0} and Lemma~\ref{vdk00} we have
\[
H_2(\GL_2(R))\simeq H_2(\GL_1(R))\oplus K_2^{M}(R).
\]
Now by the homology stability result from (i), we have
\begin{align*}
K_2^{M}(R)& \simeq H_2(\GL_2(R))/H_2(\GL_1(R))\\
& \simeq H_2(\GL(R))/H_2(\GL_1(R))\\
& \simeq H_2(\SL(R))\\
&= K_2(R).
\end{align*}
Moreover, clearly $K_2^{M}(R) \arr H_2(\GL_2(R))$ coincides with the composition
\[
K_2^{M}(R) \arr H_2(T_2) \arr H_2(\GL_2(R))
\]
and by the proof of Corollary \ref{vdk--0}, this is given by
\begin{align*}
 \{a,b\} & \mt {\bf c}\bigg({\mtx a 0 0 1},{\mtx b 0 0 1}\bigg)-
{\bf c}\bigg({\mtx a 0 0 1},{\mtx 1 0 0 b}\bigg)\\
& ={\bf c}\bigg({\mtx a 0 0 1},{\mtx b 0 0 {b^{-1}}}\bigg).
\end{align*}
\par (iii) The last claim follows from (ii) and Lemma \ref{vdk00}.
\end{proof}

%\begin{rem}
%We are not sure, but it seems that the above proposition for local rings with 
%$|R/\mmm_R|=5$ is new!
%\end{rem}

\begin{exa}
Since $H_2(\GL_2(\F_4))\simeq \z/2$ (see Example \ref{sln-finite-fields}) and
\[
H_2(\GL(\F_4))\simeq H_2(\GL_1(\F_4)) \oplus K_2(\F_4)=0,
\]   
the natural map $H_2(\GL_2(\F_4))\arr H_2(\GL(\F_4))$ 
is surjective but not injective. But $K_2^M(\F_4)=K_2(\F_4)=0$.
\end{exa}

\begin{rem}
(i) Dennis has proved that for any local ring $R$ the natural map $K_2^M(R)\arr K_2(R)$
is surjective \cite[Theorem~2.5]{stein1973}. His result also is true for any semilocal 
ring which has at most one residue field with two elements.
\par (ii) It is a result of Van der Kallen that $K_2^M(R)\arr K_2(R)$ is an isomorphism 
for any local ring such that its residue field has more than five elements. In fact,
he proved this for a larger class of rings called $5$-fold stable rings 
\cite[5.2]{vdkallen1977}. 
%The above proof, in the case of local rings, is different than 
%Van der Kallen's original proof and is in the line of \cite{hutchinson1990}, \cite{suslin1991}
%and \cite{mirzaii2011}.
%\par (iii) In Proposition \ref{mir-stability} below, we will show that $K_2^M(R)\simeq K_2(R)$
%even for any local rings such that its residue field has more than four elements. 
%%Our arguments
%%also works for any semilocal ring, where its residue fields have more than four elements. 
%%Apparently this is a new results, but we are not sure about it.
%\par (iii) If $R=\F_2[[x]]$, then $K_2^M(R)\arr K_2(R)$ is surjective but is not injective
%\cite[A1.4]{kahn1993}.
\end{rem}

%%%%%%%%%%%%%%%%%%%%%%%%%%%%%%%%%%%%%%%%%%%%%%%%%%%%%%%%%%%%%%%%%%%%%%%%%%%%%%%%%%%%%%%%%
\section{Stability for the third homology group and the third {\it K}-group}\label{H3}
%%%%%%%%%%%%%%%%%%%%%%%%%%%%%%%%%%%%%%%%%%%%%%%%%%%%%%%%%%%%%%%%%%%%%%%%%%%%%%%%%%%%%%%%%

To prove the Bloch-Wigner exact sequence we will need to show that the map
$H_3(\GL_3(R)) \arr H_3(\GL_{4}(R))$ is an isomorphism. We prove this when
the residue field of $R$ has sufficient elements.

Let $S=\{w_1, \dots ,w_m\}$ be a subset of $R^n$. We say that $(v_1, \dots ,v_l)$ 
is in general position with $S$ if for all $a_i \in R$, $i=1, \dots, l+m$ and only 
$n$ many of them non-vanishing, 
\[
a_1v_1 + \cdots + a_l v_l + a_{l+1} w_1 + \cdots + a_{l+m} w_m=0
\]
implies that $a_i=0$ for all $i=1, \dots, l$.

\begin{dfn}
Let $c_n(R)$ be the minimum of the numbers $|S|$ where $S$ is a finite subset 
of $R^n$ such that there is no further element in $R^n$ which is in general 
position to $S$.
\end{dfn}

\begin{rem}
We believe that $c_n(R)$ is equal to the largest natural number such that 
there is a finite subset $T$ of $R^n$ such that the element of $T$ are in general position. 
This can be seen by direct calculations for local rings with small residue fields.
\end{rem}

\begin{exa}\label{cnR}
If $v_1, \dots , v_n, v_{n+1}, \dots, v_l\in R^n$ are in general position, then by 
multiplying the above vectors with a suitable invertible matrix we may assume that 
$v_1=e_1, \dots, v_n=e_n$. Thus $v_{n+1}=a_1e_1+\dots+a_ne_n$, where for any $i$, 
$a_i\neq 0$. Replacing $e_i$ with $a_ie_i$ for $1\leq i \leq n$ 
and multiplying them with a suitable invertible matrix,
we may assume that $v_1=e_1, \dots, v_n=e_n, v_{n+1}=e_1+\dots+e_n$. Moreover, in a 
similar way we may assume that $v_{n+j}=e_1+a_{2,j}e_2+\dots+a_{n,j}e_n$ for $2\leq j\leq l-n$.
\par (i) The elements of the set $\{e_1,e_2, e_1+e_2,e_1+a_ie_2| \ i=1, \dots,m\}\se R^2$ 
are in general position if and only if $a_i,1-a_i, a_i-a_j\in \rr$, where $i\neq j$. Thus
\[
c_2(R)=|\PP^1(R/\mmm_R)|=|R/\mmm_R|+1.
\]
\par (ii) The above argument shows that  
$c_n(R)\geq n+1$. If $|R/\mmm_R|\leq n$, then clearly $c_n(R)=n+1$.

\par (iii) If $|R/\mmm_R|=\infty$, then it is not difficult to show that for any finite
set $S$ of elements in general position of $R^n$, one can always find many vectors that are 
in general position with $S$. Therefore $c_n(R)=\infty$.

\par(iv) The elements of the set
\[
\{ e_1, e_2, e_3, e_1+e_2+e_3, e_1+a_ie_2+b_ie_3 \mid i=1, \dots, m \} \se R^3
\]
are in general position if and only if 
\[
a_i,1- a_i, b_i, 1- b_i, a_i- a_j, b_i- b_j,a_i- b_i,
x_{ij}:=a_ib_j-a_jb_i,x_{ij}-x_{ik}+x_{jk} \in R^\times,
\]
%, $i\neq j$, and $x_{i,j}-x_{i,k}+x_{j,k} \in F^\times$,
where $i,j,k$ are distinct. By direct computations we get 
\begin{align*}
c_3(R)&=\begin{cases}
4& \text{if $|R/\mmm_R|=2, 3$}\\
6& \text{if $|R/\mmm_R|=4, 5$}\\
8& \text{if $|R/\mmm_R|=7, 8, 9$}
\end{cases}
\\
&%& \text{and} 
\\
c_3(R)&\geq 9 \ \ \ \ \ \ \text{if $|R/\mmm_R| > 9$}. 
\end{align*}
\par (v) By direct computations one can show that
\begin{align*}
c_4(R)&=\begin{cases}
5& \text{if $|R/\mmm_R|=2, 3, 4$}\\
8& \text{if $|R/\mmm_R|=5$}
\end{cases}
\\
&%& \text{and} 
\\
c_4(R)&\geq 9 \ \ \ \ \ \ \text{if $|R/\mmm_R| \geq 7$}. 
\end{align*}
Note that in $\F_7^4$, the elements of the set $\{e_1,e_2,e_3,e_4,e_1+e_2+e_3+e_4,
e_1+2e_2+3e_3+4e_4, e_1+5e_2+6e_3+2e_4, e_1+3e_2+4e_3+5e_4, e_1+6e_2+2e_3+3e_4\}$
are in general position.
\end{exa}

For a non-negative integer $l$, 
let $\CC_l(R^n)$ be the free abelian group with a basis consisting of 
$(l+1)$-tuples $(v_0, \dots, v_l)$, where $v_i\in R^n-\{0\}$ and for any $k\leq l+1$,
the submodule of $R^n$ generated by any $k$ elements of the set $\{v_0, \dots, v_l\}$
is a free summand of $R^n$. We define the differential operators
$\tilde{\partial}_l: \CC_l(R^n) \arr \CC_{l-1}(R^n)$, $l\ge 0$,
in the usual way and consider $\CC_l(R^n)$ as a left $\GL_n(R)$-module in a natural way. 
Note that $\CC_{-1}(R^n):=\z$. So we have the complex of $\GL_n(R)$-modules:
\begin{align*}
\CC_\bullet(R^n):& \ \ \ 
\cdots  \arr \CC_{2}(R^n) \arr \CC_{1}(R^n)  \arr \CC_0(R^n) \arr \CC_{-1}(R^n) \arr 0.
\end{align*}
For a finite subset $S$ of $R^n$, let $\CC_\bullet(R^n, S)$ be the subcomplex of
$\CC_\bullet(R^n)$, where $\CC_l(R^n, S)$ is the free abelian group
generated by terms $(v_0, \dots, v_l)$ which are in general position with 
$S$. Note that $\CC_{-1}(R^n, S)=\z$ and $\CC_\bullet(R^n, \emptyset)=\CC_\bullet(R^n)$.
%It is easy to see that the sequence
%$\CC_1(R^n) \overset{\tilde{\partial}_1}{\larr}
%\CC_0(R^n) \overset{\tilde{\epsilon}}{\larr} \z \arr 0$
%is exact. In fact we can prove more. 

\begin{lem}\label{sah-guin}
For a finite subset $S$ of $R^n$, the complex $\CC_\bullet(R^n,S)$ is exact for 
$-1 \leq  m \leq (c_n(R)-|S|-3)/2$. In particular, $\CC_\bullet(R^n)$ is exact 
for $-1 \leq  m \leq (c_n(R)-3)/2$. If $R$ is a field, then $\CC_\bullet(R^n)$ 
is exact.
\end{lem}
\begin{proof}
%The first part can be proved as the proof of \cite[Theorem~3]{kerz2005}. 
The first part can be proved similar to \cite[Theorem~3]{kerz2005}. 
The proof is by induction on $m$. If $m=-1$, then
from $-1\leq (c_n(R)-|S|-3)/2$, we obtain $|S|<c_n(R)$.
So there is $v\in R^n$ in general position with $S$. Thus $\tilde{\epsilon}(n(v))=n$.
This shows that $\CC_0(R^n,S)\arr \CC_{-1}(R^n, S)=\z$ is surjective.
Now let $m=0$. Then $|S|+2 <c_n(R)$. Let $z \in \ker(\tilde{\epsilon})$. We may assume that
$z=(v_0)-(v_0')$. Set $S':=S\cup \{v_0,v_0'\}$. Since $|S'|=|S|+2 <c_n(R)$, there is a $w\in R^n$
which is in general position with $S'$. Thus $y:=(w,v_0)-(w,v_0')\in \CC_1(R^n,S)$ and 
$\tilde{\partial}_1(y)=x$. 

Now let the claim is true for all numbers $<m$.
Fix a nonzero vector $v\in R^n$ which is in general position with $S$.

For $z=\sum_jl_j(v_{0,j},\dots, v_{m,j}) \in \CC_m(R^n,S)$, let $I(z)$, $-1\leq I(z)\leq m$, be 
the greatest natural number such that for any $j$, $(v_{0,j},\dots, v_{I(z),j})$
is in general position with $S_j:=S\cup \{v\}\cup \{v_{I(z)+1,j},\dots, v_{m,j}\}$.

Now suppose that $z \in \ker(\tilde{\partial}_m)$. We want to show that 
$z\in \im(\tilde{\partial}_{m+1})$.
First we show that we may assume that $I(z)\geq 0$. Let $I(z)=-1$ and consider the set 
$S_j:=S\cup\{v\}\cup \{v_{0,j},\dots, v_{m,j}\}$. Since $|S_j|=|S|+m+2<c_n(R)$
for any $j$, there is $w_j\in R^n$
which is in general position with $S_j$. Now if 
\[
z_1:=z-\tilde{\partial}_{m+1}\Big(\sum_jl_j(w_j,v_{0,j},\dots, v_{m,j})\Big),
\]
then $\tilde{\partial}_m(z_1)=0$ and $I(z_1)\geq 0$. Thus we may assume that $I(z)\geq 0$. 
Moreover, we want to show that we may assume that $I(z)=m$. So let $-1< I(z) <m$ and write $z=z'+z''$, 
where $z'$ contains those terms $y_j:=(v_{0,j},\dots, v_{m,j})$ of $z$
such that $I(y_j)=I(z)$ and $z''$ contains those terms $y_j$ of $z$ such that $I(y_j)>I(z)$.
We may write 
\[
z=\sum_k s_k x_k +z''
\]
such that $s_k\in \CC_{I(z)}(R^n)$ and 
$x_k \in \CC_{m-I(z)-1}(R^n)$. Note that the terms $x_k$ are disjoint. Thus
\[
\sum_k \tilde{\partial}_{I(z)}(s_k)x_k+(-1)^{I(z)}s_k\tilde{\partial}_{m-I(z)-1}(x_k)+
\tilde{\partial}_{m}(z'')=0.
\]
We show that the only terms $y_k'$ in the left hand side of the above formula with $I(y_k')<I(z)$ are 
the terms of $\sum_k \tilde{\partial}_{I(z)}(s_k)x_k$. Clearly the terms of
$\tilde{\partial}_{m}(z'')$ and $s_k\tilde{\partial}_{m-I(z)-1}(x_k)$ do not 
have this property. Now assume that 
\[
y_k':=(v_{0,k}, \dots ,\widehat{v_{i,k}},\dots, v_{I(z),k},v_{I(z)+1,k},\dots, v_{m,k}), 
\]
$0\leq i\leq I(z)$, is a term of $\tilde{\partial}_{I(z)}(s_k)x_k$.
If $I(y_k')=I(z)$, then the vector $v_{I(z)+1,k}$ is in general position with
$S\cup \{v\}\cup \{v_{I(z)+2,k},\dots, v_{m,k}\}$. But 
$(v_{0,k},\dots, v_{I(z),k})$ is in general position with 
$S\cup \{v\}\cup \{v_{I(z)+1,k},\dots, v_{m,k}\}$. This implies that 
$(v_{0,k}, \dots, v_{I(z),k}, v_{I(z)+1,k})$
should be in general position with the set $S \cup \{v\}\cup \{v_{I(z)+2,k},\dots, v_{m,k}\}$,
which is not possible, because $(v_{0,k},\dots, v_{m,k})$ is a term of $z'$.

Since $x_k$'s are disjoint, we have $\tilde{\partial}_{I(z)}(s_k)=0$. If 
$x_k\!=\!(w_{1,k},\dots, w_{m-I(z),k})$, then clearly
the terms of $s_k$ are in general position with the set
$S_k''':=S\cup \{v\} \cup\{w_{1,k},\dots, w_{m-I(z),k}\}$. Thus
$s_k\in \CC_{I(z)}(R^n,S_k''')$. Moreover, we have
\[
I(z)\leq \frac{c_n(R)-(|S|+m-I(z)+1)-3}{2}.
\]
So by induction, there is $s_k'\in \CC_{I(z)+1}(R^n,S_k''') $ such that 
$\tilde{\partial}_{I(z)+1}(s_k')=s_k$, for any $k$. Now if we put 
$z_1:=z-\sum_k\tilde{\partial}_{I(z)+1}(s_k'x_k)$, 
then $z_1\in \ker(\tilde{\partial}_{m})$ and 
\[
I(z_1)=I\bigg(z''+(-1)^{I(z)+1}\sum_ks_k'\tilde{\partial}_{m-I(z)-1}(x_k)\bigg)>I(z).
\]
By continuing this process we may assume that $I(z)=m$. This means that the vector $v$ is in
general position with the set $S\cup \{v_{0,k},\dots, v_{m,k}\}$ for all $k$. Thus
if $Z:=\sum_kl_k(v, v_{0,k},\dots, v_{m,k})$, then $Z\in \CC_{m+1}(R^n,S)$ and we have
$\tilde{\partial}_{m+1}(Z)=z$.

If $R=F$ is a field, then the proof of the claim is very easy. Let 
$z_1=\sum_jl_j(v_{0,j},\dots, v_{m,j}) \in \CC_m(F^n)$ and let $v$ be any nonzero
vector of $F^n$. Then $Z_1:=\sum_jl_j(v, v_{0,j},\dots, v_{m,j})\in \CC_{m+1}(F^n)$.
Now if $z_1 \in \ker(\tilde{\partial}_m)$, then $\tilde{\partial}_{m+1}(Z_1)=z_1$.
\end{proof}

\begin{rem}
One can show that in general $\CC_\bullet(R^n)$ is exact for $-1 \leq  m \leq n-2$.
But we need the exactness of this complex for $m\leq n-1$ (at least when $n=3$).
As we have seen in the above lemma this is easy when $R$ is a filed.
This is probably true for any local ring, but I do not know how to resolve this problem 
without considering $c_n(R)$. Once this is done, we may remove the condition
about $c_{n+1}(R)$ from the following theorem.
\end{rem}

\begin{thm}\label{sah-stability}
Let $R$ be a local ring. If $R/\mmm_R$ is finite we assume that $|R/\mmm_R|=p^d$.
\par {\rm (i)} If $R$ is a domain or an algebra over a field and 
$n \leq \min\{(p-1)d-1,(c_{n+1}(R)-3)/2\}$, then
\[
H_n(\GL_n(R)) \overset{\simeq}{\larr} H_n(\GL_{n+1}(R)) \overset{\simeq}{\larr} 
H_n(\GL_{n+2}(R)) \overset{\simeq}{\larr}\cdots.
\]
If $R$ is a field we may only assume that $n \leq (p-1)d-1$.
\par {\rm (ii) } If $n \leq \min\{(p-1)d-3, (c_{n+1}(R)-3)/2\}$, then
\[
H_n(\GL_n(R)) \overset{\simeq}{\larr} H_n(\GL_{n+1}(R)) \overset{\simeq}{\larr} 
H_n(\GL_{n+2}(R)) \overset{\simeq}{\larr}\cdots.
\]
\end{thm}
\begin{proof}
This can be done as the proof of Theorem~1 in  Subsection 2.2 of \cite{guin1989}. 
\end{proof}

\begin{prp}\label{h3-k3}
Let $|R/\mmm_R|\neq 2,3,4,5,8,9,16,32$. If $R$ is a domain or an algebra over a field,
we may only assume that $|R/\mmm_R|\neq 2,3,4,5,8$. Then we have the homology stability
\[
H_3(\GL_3(R)) \overset{\simeq}{\larr} H_3(\GL_{4}(R)) \overset{\simeq}{\larr} 
H_3(\GL_{5}(R)) \overset{\simeq}{\larr}\cdots.
\]
Furthermore, $K_3^\ind(R)\simeq H_3(\SL(R))/T$, were $T$ is generated by the elements
${\bf c}(\diag(a,1,a^{-1}), \diag(b,b^{-1},1), \diag(c,1,c^{-1}))$, $a,b,c\in \rr$. 
\end{prp}
\begin{proof}
The homology stability result follows from Theorem \ref{sah-stability} and Example \ref{cnR}(v).
The second claim can be proved as \cite[Corollary 2.4, \S2]{mirzaii-mokari2015}. Here
we need the homology stability to generalize \cite[Proposition 2.1]{mirzaii-mokari2015}.
\end{proof}

%%%%%%%%%%%%%%%%%%%%%%%%%%%%%%%%%%%%%%%%%%%%%%%%%%%%%%%%%%%%%%%%%%%%%%%%%%%%%%%%%%%%%%
\section{Third homology of general linear groups of rank 2 and 3}\label{bloch-GL2}
%%%%%%%%%%%%%%%%%%%%%%%%%%%%%%%%%%%%%%%%%%%%%%%%%%%%%%%%%%%%%%%%%%%%%%%%%%%%%%%%%%%%%%

Let $R$ be a local ring such that its residue field has at least four elements.
Then $H_n(\GL_2(R), C_\bullet(R^2))\simeq H_n(\GL_2(R))$ for $0 \leq n \leq 3$.
The spectral sequence $E_{p,q}^1$ gives a filtration 
\[
0=F_{-1}H_3(\GL_2(R)) \se \dots \se F_{3}H_3(\GL_2(R))=H_3(\GL_2(R)),
\]
such that $E_{i,3-i}^\infty\simeq F_{i}H_3(\GL_2(R))/F_{i-1}H_3(\GL_2(R))$.
By Lemma \ref{d3}, we have
\[
B(R)\simeq E_{3,0}^\infty\simeq H_3(\GL_2(R))/F_{2}H_3(\GL_2(R)).
\]
Since $E_{2,1}^\infty=0$, $F_{1}H_3(\GL_2(R))=F_{2}H_3(\GL_2(R))$. Moreover,
\[
\!\!\!\!\!\!\!\!
E_{0,3}^\infty\simeq F_{0}H_3(\GL_2(R))=\im(H_3(B_2)\arr H_3(\GL_2(R))).
\] 
The next lemma studies the map 
\[
(\rr \otimes \rr)^\sigma\!\!=\!E_{1,2}^2 -\!\!\!\two\! E_{1,2}^\infty\!=
\!F_{2}H_3(\GL_2(R))/H_3(B_2)
\! \se \! H_3(\GL_2(R))/H_3(B_2).
\]

\begin{lem}\label{b-h00}
Let $u \in (\rr \otimes \rr)^\sigma \se H_2(T_2)^\sigma$ and  $h \in \mathcal{B}_2(T_2)_{T_2}$ 
a representing cycle for $u$. Let $\tau$ be the automorphism of $\mathcal{B}_\bullet(T_2)_{T_2}$ 
induced by $\sigma$ and let $\tau(h)-h=\partial_3^{T_2}(b)$, $b \in \mathcal{B}_3(T_2)_{T_2}$.
Then the image of $u$ under the map
\[
(\rr \otimes \rr)^\sigma \arr H_3(\GL_2(R))/H_3(B_2)
\] 
coincides with the homology class of the cycle $b -\rho_s(h)$, where $s:= {\mtx 0 1 1 0}$ and
\[
\rho_s([g_1|g_2]):= [s| sg_1s^{-1}|sg_2s^{-1}]-[g_1|s|sg_2s^{-1}]+[g_1|g_2|s].
\]
\end{lem}
\begin{proof}
See \cite[Lemma 2.5]{suslin1991}.
\end{proof}

Let $\GM_2(R)$ denotes the group of monomial matrices in $\GL_2(R)$ and consider the extension
\[
1 \arr T_2 \arr \GM_2(R) \arr \Sigma_2 \arr 1,
\]
where $\Sigma_2=\Bigg\{ {\mtx 1 0 0 1}, {\mtx 0 1 1 0} \Bigg\}$. We often think of
$\Sigma_2$ as the symmetric group of order two $\{1, \sigma\}$. Note that
$\GM_2(R) = T_2 \rtimes \Sigma_2$ and the action of $\sigma$ on $T_2$ is given by
$\sigma(a,b)=(b,a)$. From this extension we obtain the first quadrant spectral sequence
\[
{E'}^2_{p,q}=H_p(\Sigma_2 ,H_q(T_2)) \Rightarrow H_{p+q}(\GM_2(R)).
\]
This spectral sequence gives us a filtration
\[
0=F_{-1}H_3(\GM_2(R)) \se \dots \se F_3H_3(\GM_2(R))=H_3(\GM_2(R)),
\]
such that
\[
\begin{array}{l}
\vspace{1.5 mm}
{E'}^\infty_{0, 3}\simeq F_0H_3(\GM_2(R))=H_3(T_2)_\sigma,\\
\vspace{1.5 mm}
{E'}^\infty_{1, 2}\simeq {E'}_{1,2}^2 \simeq F_1H_3(\GM_2(R))/F_0H_3(\GM_2(R)),\\
\vspace{1.5 mm}
{E'}^\infty_{2, 1}\simeq F_2H_3(\GM_2(R))/F_1H_3(\GM_2(R))=0,\\
{E'}^\infty_{3, 0}\simeq 
H_3(\GM_2(R))/F_2H_3(\GM_2(R)).
\end{array}
\]
From the natural inclusion $\Sigma_2 \se \GM_2(R)$, one easily sees that the
composition $H_3(\Sigma_2) \arr H_3(\GM_2(R)) \arr H_3(\Sigma_2)$ coincides with the
identity map. Thus ${E'}^\infty_{3, 0}\simeq H_3(\Sigma_2)$.
Now the above relations imply the following isomorphisms
\begin{align}\label{GM2}
H_3(\GM_2(R)) \simeq F_2H_3(\GM_2(R)) \oplus H_3(\Sigma_2),
\end{align}
\begin{align}
{E'}_{1, 2}^2 \simeq F_2H_3(\GM_2(R))/H_3(T_2).
\end{align}

The next lemma gives an explicit description of the composition
\[
\begin{CD}
H_2(T_2)^\sigma -\!\!\!\twoheadrightarrow {E'}_{1,2}^2 \overset{\simeq}{\larr}
F_2H_3(\GM_2(R))/H_3(T_2) \se H_3(\GM_2(R))/H_3(T_2).
\end{CD}
\]

\begin{lem}\label{b-h}
Let $u \in (\rr \otimes \rr)^\sigma \se H_2(T_2)^\sigma$ and  $h \in \mathcal{B}_2(T_2)_{T_2}$ a 
representing cycle for $u$. Let $\tau$ be the automorphism of $\mathcal{B}_\bullet(T_2)_{T_2}$ 
induced by $\sigma$ and let $\tau(h)-h=\partial_3^{T_2}(b)$, $b \in \mathcal{B}_3(T_2)_{T_2}$. 
Then the image of $u$ under the map
\[
(\rr \otimes \rr)^\sigma \arr H_3(\GM_2(R))/H_3(T_2)
\] 
coincides with the homology class of the cycle $b -\rho_s(h)$, where $\rho_s$ is defined
in Lemma \ref{b-h00}.
\end{lem}
\begin{proof}
See \cite[Lemma~4.3]{mirzaii-mokari2015}.
\end{proof}

For more details about the spectral sequence ${E'}_{p,q}^1$ see  \cite[Section~4]{mirzaii-mokari2015}.
Since the residue field of $R$ has at least four elements, $H_i(C_\bullet(R^2))=0$
for $-1\leq i \leq 3$. From this we obtain a natural map
\[
\varphi: H_3(\GL_2(R)) \arr H_3(C_\bullet(R^2)_{\GL_2(R))})=\ppp(R).
\]
%The following theorem and Proposition we will study this map.

\begin{thm}\label{gl-gm}
Let $|R/\mmm_R|\neq 2, 3, 4, 5, 8, 9, 16, 32$. If $R$ is a domain 
or an algebra over a field, we may only assume that
$|R/\mmm_R|\neq 2, 3, 4, 8$. Then we have the exact sequence
\[
H_3(\GM_2(R)) \arr H_3(\GL_2(R)) \overset{\varphi}{\arr} B(R) \arr 0.
\]
\end{thm}
\begin{proof}
By Proposition \ref{affine}, $H_i(B_2)\simeq H_i(T_2)$ for $0 \leq i\leq 3$.
Now consider the natural map $H_3(\GM_2(R))\arr H_3(\GL_2(R))$. It is clear
from the filtrations of $H_3(\GM_2(R))$ and $H_3(\GL_2(R))$ that
$F_0H_3(\GM_2(R))$ maps onto $F_0H_3(\GL_2(R))$. This fact together with Lemmas \ref{b-h}
and \ref{b-h00} imply that $F_2H_3(\GM_2(R))$ maps onto $F_2H_3(\GL_2(R))$. Thus 
\[
\im(H_3(\GM_2(R)))= F_2H_3(\GL_2(R)) +\im H_3(\Sigma_2).
\]
The matrix $s={\mtx 0 1 1 0}$ is conjugate to the matrix ${\mtx 1 1 0 {-1}} \in B_2$. Hence
\begin{align*}
\im(H_3(\Sigma_2)\arr H_3(\GL_2(R))) & \se \im(H_3(B_2)\arr H_3(\GL_2(R))) \\
& = F_0H_3(\GL_2(R)\se F_2H_3(\GL_2(R)).
\end{align*}
This completes the proof of the theorem.
\end{proof}

\begin{thm}\label{kernel}
Let $|R/\mmm_R|\neq 2, 3, 4, 5, 8, 9, 16, 32$. 
If $R$ is a domain or an algebra over a field, we only may assume that
$|R/\mmm_R|\neq 2, 3, 4, 8$. Then the kernel of
$\inc_\ast:H_3(\GL_2(R)) \arr H_3(\GL_3(R))$  consists of elements of the form
${\begin{array}{c} \!\! \sum \!\! \end{array}}
{\rm \bf{c}} (\diag(a,1), \diag(1,b),\diag(c, c^{-1}))$
provided that
\[
{\begin{array}{c} \!\! \sum \!\! \end{array}}
a\otimes \{b, c\}+b\otimes \{a, c\}=0 \in \fff \otimes K_2^M(R).
\]
In particular, $\ker(\inc_\ast) \subseteq \fff\cup H_2(\GL_1(R)) \subseteq H_3(\GL_2(R))$,
where the cup product is induced by the inclusion
$\inc: \fff \times \GL_1(R) \arr \GL_2(R)$. Moreover $\ker(\inc_\ast)$ is a $2$-torsion
group.
\end{thm}
\begin{proof}
This can be proved as \cite[Theorem~3.1]{mirzaii2012} using Lemma \ref{mir-elb2}(ii).
\end{proof}

Let  $C_l(R^n)$ be the free abelian groups with a basis consisting of 
$(l+1)$-tuples $(\lan w_0\ran , \dots, \lan w_l\ran )$, where every
$\min\{l+1, n\}$ of  $w_i \in R^n$ are basis of a free direct summand of $R^n$.
Let us define differentials $\partial_l : C_l(R^n) \arr C_{l-1}(R^n)$, 
$l\ge 0$, in usual way. So we have the complex of $\GL_n(R)$-modules:
\begin{align*}
C_\bullet(R^n):& \ \
\cdots  \arr C_{2}(R^n) \arr C_{1}(R^n) \arr C_0(R^n) \arr C_{-1}(R^n)=\z \arr 0.
\end{align*}

\begin{lem}\label{kerz1}
The complex $C_\bullet(R^n)$ is exact for $-1 \leq  i \leq (c_n(R)-3)/2$.
\end{lem}
\begin{proof}
This can be proved as Lemma \ref{sah-guin}.
\end{proof}

In particular, this implies that if $|R/\mmm_R|>9$, then 
$C_\bullet(R^3)$ is exact for $-1 \leq  i \leq 3$ (see Example~\ref{cnR}).
Now as in \cite[\S3]{suslin1991} we can construct a map
\[
\rho: H_3(\GL_3(R)) \arr \ppp(R)
\]
such that the composition $H_3(\GL_2(R)) \arr H_3(\GL_3(R)) \overset{\rho}{\arr} \ppp(R)$
coincides with the map $\varphi:H_3(\GL_2(R)) \arr \ppp(R)$ constructed in above. 

\begin{thm}\label{gl-gm3}
Let $R$ be a local ring such that $|R/\mmm_R|>9$ and $|R/\mmm_R|\neq 16,32$.
%$|R/\mmm_R|\neq 2, 3,4,5,7, 8,9,16,32$.
If $R$ is a domain or is an algebra over a field we may only assume that $|R/\mmm_R|> 9$.
%Let $R$ be a local ring such that its residue field has more than nine elements.
Then we have the exact sequence
\[
H_3(\GM_2(R))\oplus H_3(T_3) \arr H_3(\GL_3(R)) \overset{\rho}{\arr} B(R) \arr 0.
\]
\end{thm}
\begin{proof}
Let $T_3:=\rr^3=\rr \times \rr \times \rr$ embeds diagonally in $\GL_3(R)$. Then by
Lemma \ref{mir-elb2}(ii) it is easy to show that $H_3(\GL_3(R))$ is generated by the
images of $H_3(\GL_2(R))$ and $\rr\cup \rr \cup \rr=\im((\rr)^{\otimes 3} \arr H_3(\GL_3(R))$. 
Now one can use Theorem \ref{gl-gm} to prove the claim as it is done in \cite[\S3]{suslin1991}. 
\end{proof}

%%%%%%%%%%%%%%%%%%%%%%%%%%%%%%%%%%%%%%%%%%%%%%%%%%%%%%%%%%%%%%%%%%%%%%%%%%%%%%%%%%%%%%%%%%%%%%%
\section{A Bloch-wigner exact sequence}\label{B-W}
%%%%%%%%%%%%%%%%%%%%%%%%%%%%%%%%%%%%%%%%%%%%%%%%%%%%%%%%%%%%%%%%%%%%%%%%%%%%%%%%%%%%%%%%%%%%%%%

Now we are ready to formulate the main theorem of this article.

\begin{thm}[Bloch-Wigner exact sequence]\label{Bloch-wigner}
Let $R$ be a local ring such that $|R/\mmm_R|>9$ and $|R/\mmm_R|\neq 16,32$.
%$|R/\mmm_R|\neq 2, 3,4,5, 8,9,16,32$.
If $R$ is a domain or is an algebra over a field we may only assume that 
$|R/\mmm_R|> 9$. Then we have the exact sequence
\[
T_R \arr K_3^\ind(R) \arr B(R) \arr 0,
\]
where $T_R$ sits in the short exact sequence
\[
0 \arr \tors(\mu(R),\mu(R))_\sigma \arr T_R \arr
H_1(\Sigma_2, \mu_{2^\infty}(R)\otimes_\z\mu_{2^\infty}(R)) \arr 0.
\]
Moreover, if there is a homomorphism $R \arr F$, $F$ a field, such that the map 
$\mu(R) \arr \mu(F)$ is injective, then we have the exact sequence
\[
0 \arr \tors(\mu(R), \mu(R))^\sim \arr K_3^\ind(R) \arr B(R) \arr 0,
\]
where 
%the group $\tors(\mu(R),\mu(R))^\sim$ is the unique nontrivial extension of
%$\tors(\mu(R),\mu(R))$ by $\z/2$ if $\char(F)\neq 2$ and is equal to 
%$\tors(\mu(R),\mu(R))$ if $\char(F)= 2$. Furthermore 
the composition 
\[
\tors(\mu(R),\mu(R)) \arr \tors(\mu(R), \mu(R))^\sim \arr K_3^\ind(R)
\] 
is induced by the map $\mu(R) \arr \SL_2(R)$,  $\xi \mt \diag(\xi, \xi^{-1})$.
\end{thm}
\begin{proof}
This can be done as the proof of \cite[Theorem~5.1]{mirzaii-mokari2015}
\end{proof}

Note that in the above theorem, $T_R$ can be defined as
\[
T_R:=F_2H_3(\GM_2(R))/H_3(T_2)_\sigma=F_2H_3(\GM_2(R))/H_3(T_2),
\]
and we have the following lemma.

\begin{lem}\label{M-K}
Let $R$ be any commutative ring. Then $T_R$ sits in the short exact sequence
\[
0 \arr \tors(\mu(R),\mu(R))_\sigma \arr T_R \arr
H_1(\Sigma_2, \mu_{2^\infty}(R)\otimes \mu_{2^\infty}(R)) \arr 0.
\]
Moreover, if we have a homomorphism $R \arr F$, $F$ a field, such that 
$\mu(R) \arr \mu(F)$ is injective, then $T_R \simeq \tors(\mu(R),\mu(R))^\sim$.
\end{lem}
\begin{proof}
See \cite[Section~4, Lemma~4.4, Corollary~4.5]{mirzaii-mokari2015}.
\end{proof}

\begin{cor}\label{B-B}
Let $R$ be a discrete valuation ring with field of fraction $K$ and 
residue field $F=R/\mmm_R$ with more than nine elements. Suppose that 
either $\char(K)=\char(F)$ or $F$ is finite. Then the natural map 
$B(R) \arr B(K)$ is an isomorphism.
\end{cor}
\begin{proof}
By \cite[Theorem 2.1]{hutchinson2014} $K_3^\ind(R)\simeq  K_3^\ind(K)$. Thus the 
claim follows from the Bloch-Wigner exact sequence for $R$ and $K$, obtained from 
Theorem~\ref{Bloch-wigner}. Note that here $\mu(R)=\mu(K)$.
\end{proof}

If $F$ is a finite field such that $|F|>9$, then by Theorem \ref{Bloch-wigner}
we have the Bloch-Wigner exact sequence
\[
0 \arr \tors(\mu(F), \mu(F))^\sim \arr K_3^\ind(F) \arr B(F) \arr 0.
\]
But by a direct approach we can prove a better result with easier arguments.

\begin{prp}\label{mirzaii3}
Let $F$ be a finite field with at least four elements. 
\par {\rm (i)} If $F\neq \F_4,\F_8$, then we have the exact sequence
\[
0 \arr \tors(\mu(F),\mu(F))^\sim \arr H_3(\SL_2(F))_{F^\times} \arr B(F) \arr 0.
\]
If $F=\F_4$ or $\F_8$, then we have the exact sequence
 \[
0 \arr \tors(\mu(F),\mu(F))\oplus \z/2 \arr H_3(\SL_2(F))_{F^\times} 
\arr B(F) \arr 0.
\]
\par {\rm (ii)} If $F\neq \F_4,\F_8$, then 
$H_3(\SL_2(F))_{F^\times}\simeq  K_3^\ind(F)$ and if $F=\F_4$ or $\F_8$
then $H_3(\SL_2(F))_{F^\times}\simeq  K_3^\ind(F)\oplus \z/2$.
Moreover, we have the Bloch-Wigner exact sequence
\[
0 \arr \tors(\mu(F), \mu(F))^\sim \arr K_3^\ind(F) \arr B(F) \arr 0.
\]
\end{prp}
\begin{proof}
%Let $F\neq \F_2,\F_3$.
Since $F^\times$ is a finite cyclic group, $H_2(\GL_1(F))=0$. 
%Thus $\H_3(\SL_2(F))=H_3(\GL_2(F))/H_3(\GL_1(F))$. 
If $|F|> 4$, we have $H_1(\SL_2(F))=0$, $H_2(\SL_2(F))_{F^\times}=0$ 
(see Example \ref{sln-finite-fields}). Hence by Lemma \ref{dwyer},
$H_r(F^\times, H_s(\SL_2(F)))=0$ for all $r\geq 0$ and for $s=1,2$. 
If $F=\F_4$,  we have $H_1(\SL_2(F))=0$ and $H_2(\SL_2(F))=\z/2$.
Since $\F_4^\times={\F_4^\times}^2$, the action of $F^\times$ on 
$H_s(\SL_2(F))$ is trivial. 
These together  with the universal coefficient theorem imply that 
$H_r(F^\times, H_s(\SL_2(F))=0$ for any $r \geq 1$ and $s=1,2$.
Now by an easy analysis of the corresponding Lyndon-Hochschild-Serre 
spectral sequence of the extension 
\[
1 \arr  \SL_2(F) \arr \GL_2(F) \arr F^\times \arr 1,
\]
%for any $F$ with at least four elements, 
we obtain the isomorphism  
\[
H_3(\SL_2(F))_{F^\times}\simeq H_3(\GL_2(F))/H_3(\GL_1(F)). 
\]

\par (i) First let $F\neq \F_4, \F_8$.
From the proof of Theorem \ref{gl-gm} we see that we have
the exact sequence
\[
F_2H_3(\GM_2(F)) \arr H_3(\GL_2(F)) \arr B(F) \arr 0.
\]
Let $M:=H_3(F^\times) \oplus H_3(F^\times) \se H_3(T_2)$. From the commutative diagram
\[
\begin{CD}
M & \hspace{-6 mm} \overset{=}{-\!\!\!-\!\!\!-\!\!\!-\!\!\!-\!\!\!\larr} & M  & \\
@VVV  @VVV        &        \\
F_2H_3(\GM_2(F)) & \!\!\!\larr &\!\! H_3(\GL_2(F)) & \arr  & B(F) & \arr 0,
\end{CD}
\]
we obtain the exact sequence $T_F \arr H_3(\SL_2(F))_{F^\times} \arr B(F) \arr 0$.
By Lemma \ref{M-K}, $T_F\simeq \tors(\mu(F),\mu(F))^\sim$. To finish the proof of 
the proposition we have to prove that the map
\[
\tors(\mu(F), \mu(F))^\sim \arr H_3(\SL_2(F))_{F^\times}
\]
is injective.  Since 
\[
\tors(\mu(F), \mu(F))^\sim \arr \tors(\mu(\overline{F}), \mu(\overline{F}))^\sim
=\tors(\mu(\overline{F}), \mu(\overline{F}))
\]
is injective, it is sufficient to prove that
\[
\tors(\mu(\overline{F}), \mu(\overline{F}))^\sim \arr H_3(\SL_2(\overline{F}))
\]
is injective. Since $H_3(\SL_2(\overline{F}))\simeq K_3^\ind(\overline{F})$
\cite[Corollary~5.4]{mirzaii2011}, \cite[Proposition~6.4]{mirzaii-2008}
it is sufficient to prove that 
\[
\tors(\mu(\overline{F}), \mu(\overline{F})) \arr K_3^\ind(\overline{F})
\]
is injective. But this has been proved by Suslin \cite[Theorem 5.2]{suslin1991}, 
thus the result follows. 

Now let $F= \F_4$, or $\F_8$. Then $H_3(B_2(F))\simeq H_3(T_2(F)) \oplus \z/2$ 
(see Example \ref{finite-fields}). Now from the filtration of $H_3(\GL_2(F))$
induced by the spectral sequence $E_{p,q}^1$, we obtain the exact sequences
\[
0 \arr F_2H_3(\GL_2(F))\arr H_3(\GL_2(F))\arr B(F) \arr 0,
\]
\[
H_3(B_2(F)) \arr F_2H_3(\GL_2(F)) \arr E_{2,1}^2\arr 0, 
\]
which imply the exact sequences
\[
0 \arr F_2H_3(\GL_2(F))/M\arr H_3(\SL_2(F))_{F^\times}\arr B(F) \arr 0,
\]
\[
\tors(\mu(F), \mu(F)) \oplus \z/2 \arr F_2H_3(\GL_2(F))/M \arr E_{2,1}^2\arr 0. 
\]
By Lemmas \ref{b-h00} and \ref{b-h}, ${E'}_{2,1}^2 \arr E_{2,1}^2$ is surjective. But
\[
{E'}_{2,1}^2\simeq H_1(\Sigma_2, \mu_{2^\infty}(F) \otimes \mu_{2^\infty}(F))=0.
\] Thus we have the exact sequence
\[
\tors(\mu(F), \mu(F)) \oplus \z/2 \arr H_3(\SL_2(F))_{F^\times}\arr B(F) \arr 0.
\]
Similar to the above, we can show that the natural homomorphism
\[
\tors(\mu(F), \mu(F)) \arr H_3(\SL_2(F))_{F^\times}
\]
is injective. The map 
$\z/2 \arr H_3(\SL_2(\F_8))_{{\F_8}^\times}$ is the composition map 
\[
\begin{array}{c}
\z/2\simeq \bigwedge_\z^3 \F_8^\times \arr H_3(\F_8)\simeq H_3(N(\F_8)) \arr 
H_3(\SL_2(\F_8))_{\F_8^\times}
\end{array}
\]
and the map $\z/2 \arr H_3(\SL_2(\F_4))_{{\F_4}^\times}$ is the composition map 
\[
\begin{array}{c}
\z/2\simeq (\F_4^\times\otimes \F_4^\times)^{-\sigma} 
\simeq H_3(\F_4)_{\F_4^\times}\simeq H_3(N(\F_4)) \arr 
H_3(\SL_2(\F_4))_{\F_4^\times}
\end{array}
\]
which both maps are injective by Example \ref{p-sylow}. This completes the proof 
of (i).

(ii) Let $F\neq \F_4, \F_8$. Then by Theorem \ref{kernel}, the kernel of the map 
$\inc_\ast: H_3(\GL_2(F)) \arr H_3(\GL_3(F))$  consists of elements of the form
$x={\begin{array}{c}\!\!\sum\!\!\end{array}}{\rm\bf{c}}(\diag(a,1),\diag(1,b),\diag(c,c^{-1}))$
provided that 
\[
{\begin{array}{c} \!\! \sum \!\! \end{array}}
a\otimes \{b, c\}+b\otimes \{a, c\}=0 \in F^\times \otimes K_2^M(F).
\]
Since 
$x\in \im(H_2(F^\times)\otimes F^\times \arr H_3(\GL_2(F)))$ and $H_2({F^\times})=0$, 
the kernel of $\inc_\ast$ is trivial.  On the other hand 
from the exact sequence of Lemma \ref{mir-elb2}(ii), one sees easily that $\inc_\ast$
is surjective. These together with the homology stability of Proposition \ref{h3-k3}, 
proves that
\[
H_3(\GL_2(F)) \overset{\simeq}{\larr} H_3(\GL_3(F)) \overset{\simeq}{\larr} 
H_3(\GL_4(F)) \overset{\simeq}{\larr} \cdots.
\]
Finally, we have
\begin{align*}
H_3(\SL_2(F))_{F^\times} &\simeq H_3(\GL_2(F))/H_3(\GL_1(F))\\
&\simeq H_3(\GL(F))/H_3(\GL_1(F))\\
&\simeq H_3(\SL(F))\\
&\simeq K_3(F)\\
&=K_3^\ind(F).
\end{align*}
Now let $F= \F_4$, or $F=\F_8$. Since $F^\times={{F^\times}}^2$, we have
$H_3(\SL_2(F))_{F^\times}=H_3(\SL_2(F))$. But by a direct calculation one can show that
$H_3(\SL_2(\F_4))\simeq \z/30$ and $H_3(\SL_2(\F_8))\simeq \z/126$. Moreover, by
Quillen's calculation of the $K$-groups of finite fields, we have
$K_3^\ind(\F_q)=K_3(\F_q)\simeq \z/(q^2-1)$. Thus $K_3^\ind(\F_4)\simeq \z/15$ and 
$K_3^\ind(\F_8)\simeq \z/63$. Therefore $H_3(\SL_2(F))_{F^\times}\simeq K_3^\ind(\F_q)\oplus \z/2$.
(See Example \ref{sln-finite-fields2} below for another proof of this part.) 
Finally, the Bloch-Wigner exact sequence follows from (i).
\end{proof}

\begin{cor}
The Bloch group of a finite field $\F_q$ with at least four elements is
\[
B(\F_q)\simeq \begin{cases}
\z/(q+1) & \text{if $q$ is even}\\
\z/(\frac{q+1}{2})& \text{if $q$ is odd.}
\end{cases}
\]
\end{cor}
\begin{proof}
This follows from the Bloch-Wigner exact sequence for finite fields proved
in the previous proposition and the fact that 
\[
\tors(\mu(\F_q), \mu(\F_q))\simeq {\F_q}^\times\simeq \z/(q-1).
\]
Note that $K_3^\ind(\F_q)=K_3(\F_q)\simeq \z/(q^2-1)$. Moreover, if $q$ is even then 
$\tors(\mu(\F_q), \mu(\F_q))^\sim= \tors(\mu(\F_q), \mu(\F_q))\simeq \z/(q-1)$ and
if $q$ is odd then $\tors(\mu(\F_q), \mu(\F_q))^\sim\simeq \z/2(q-1)$.
\end{proof}

\begin{exa}\label{sln-finite-fields2}
Let $F$ be any finite field. We already have proved that 
\[
H_3(\SL_2(F))_{F^\times}\simeq H_3(\GL_2(F))/H_3(\GL_1(F)),
\]
(see the proof of Proposition \ref{mirzaii3}). In fact we proved this for
$|F|\geq 4$. But the same isomorphism can be proved easily for $|F|=2,3$.

If $\char(F)=p>0$, then by Example \ref{p-sylow}, we have
\[
(H_3(\SL_2(F))_{F^\times})_{(p)}\simeq 
\begin{cases}
\z/p & \text{if $|F|=2,3,4,8$}\\
0    & \text{otherwise.}
\end{cases}
\]
This implies that 
\[
H_3(\SL_2(F))_{F^\times} \simeq \begin{cases}
H_3(\SL_2(F), \z[1/p])_{F^\times}\oplus \z/p & \text{if $|F|=2,3,4,8$}\\
H_3(\SL_2(F), \z[1/p])_{F^\times}    & \text{otherwise.}
\end{cases}
\]
But Hutchinson has proved that for any finite field $F$, the action of $F^\times$
on $H_3(\SL_2(F), \z[1/p])$ is trivial and $H_3(\SL_2(F), \z[1/p])\simeq K_3^\ind(F)$
\cite[Lemma 3.8, Corollary~3.9]{hutchinson2013}. These imply that
\[
H_3(\SL_2(F))_{F^\times} \simeq \begin{cases}
K_3^\ind(F) \oplus \z/p & \text{if $|F|=2,3,4,8$}\\
K_3^\ind(F) & \text{otherwise.}
\end{cases}
\]
\end{exa}

%%%%%%%%%%%%%%%%%%%%%%%%%%%%%%%%%%%%%%%%%%%%%%%%%%%%%%%%%%%%%%%%%%%%%%%%%%%

\bigskip
\address{{\footnotesize
Behrooz Mirzaii,

Institute of Mathematics and Computer Sciences (ICMC),

University of Sao Paulo (USP), Sao Carlos, Brazil.

e-mail:\ bmirzaii@icmc.usp.br,
}}
\end{document}